\documentclass{amsart}

\usepackage{amssymb,amsmath,esint}
\usepackage[all]{xy}
\usepackage{graphicx}
\usepackage{enumerate}

\usepackage{colortbl}
\usepackage{xcolor}

\usepackage{amsrefs}
\colorlet{linkequation}{blue}
\usepackage[colorlinks,plainpages=true,pdfpagelabels,hypertexnames=true,colorlinks=true,pdfstartview=FitV,linkcolor=blue,citecolor=red,urlcolor=black]{hyperref}

\PassOptionsToPackage{unicode}{hyperref}
\PassOptionsToPackage{naturalnames}{hyperref}
\definecolor{dgreen}{rgb}{0,0.5,0}
\definecolor{violet}{rgb}{0.5,0,0.5}
\definecolor{dred}{rgb}{0.7,0,0}
\definecolor{ddred}{rgb}{0.5,0,0}
\definecolor{dblue}{rgb}{0,0,0.5}
\definecolor{ddblue}{rgb}{0,0,0.3}

\newtheorem{theorem}{Theorem}[section]
\newtheorem{lemma}[theorem]{Lemma}

\newtheorem{definition}[theorem]{Definition}
\newtheorem{proposition}[theorem]{Proposition}

\newtheorem{remark}[theorem]{Remark}

\numberwithin{equation}{section}

\DeclareMathOperator*{\ddiv}{div}

\newcommand{\capp}{\text{cap}_p}

\newcommand{ \mr }{ \mathbb{R} }

\newcommand{ \m }{ \mathcal{M} }

\newcommand{\integral}[3]{\int_{#1} #2 \ #3}
\newcommand{\integraL}[4]{\int_{#1}^{#2} #3 \ #4}
\newcommand{\mint}[3]{\fint_{#1} #2 \ #3}

\newcommand{\norm}[1]{\left| #1\right|}
\newcommand{\Norm}[1]{\left|\hspace{-0.2mm}\left| #1 \right|\hspace{-0.2mm}\right|}
\newcommand{\gh}[1]{\left( #1\right)}
\newcommand{\mgh}[1]{\left\{ #1\right\}}
\newcommand{\bgh}[1]{\left[ #1\right]}
\newcommand{\vgh}[1]{\left< #1\right>}

\newcommand{\OO}{\Omega}

\newcommand{\ma}{\mathbf{a}}
\newcommand{\omt}{\Omega_\mathfrak{T}}
\newcommand{\omtb}{\Omega_{\widetilde{T}}}

\newcommand{\lpwwz}{L^p (0,T; W_0^{1,p}(\OO))}

\newcommand{\ik}[1]{K_{#1}^{\lambda}}

\newcommand{\iq}[1]{Q_{#1}^{\lambda}}
\newcommand{\iqq}[1]{Q_{#1}^{\lambda,+}}
\newcommand{\pc}[1]{\left \lfloor {#1} \right \rfloor}

\newcommand{\be}{\beta}
\newcommand{\la}{\lambda}
\newcommand{\La}{\Lambda}
\newcommand{\ep}{\varepsilon}

\newcommand{\C}{\mathfrak{C}}
\newcommand{\D}{\mathfrak{D}}
\newcommand{\M}{\mathcal{M}}

\newcommand{\bb}{\mathbb{R}}

\AtBeginDocument{%
   \def\MR#1{}
}

\def\XXint#1#2#3{{\setbox0=\hbox{$#1{#2#3}{\int}$}
    \vcenter{\hbox{$#2#3$}}\kern-.5\wd0}}

\begin{document}

\title[Singular parabolic measure data problems]{Regularity estimates for singular parabolic measure data problems with sharp growth}

%    Information for first author
\author[J.-T. Park]{Jung-Tae Park}
%    Address of record for the research reported here
\address{J.-T. Park: Korea Institute for Advanced Study, Seoul 02455, Republic of Korea}
\email{ppark00@kias.re.kr}

%    Information for second author
\author[P. Shin]{Pilsoo Shin}
\address{P. Shin: Department of Mathematics, Kyonggi University, Suwon 16227, Republic of Korea}
\email{shinpilsoo.math@kgu.ac.kr}

\thanks{J.-T. Park was supported by the National Research Foundation of Korea grant (No. NRF-2019R1C1C1003844) from the Korea government and a KIAS Individual Grant (No. MG068102) from Korea Institute for Advanced Study. P. Shin was supported by the National Research Foundation of Korea grant  (No. NRF-2020R1I1A1A01066850) from the Korea government.}
%\thanks{J.-T. Park was supported by NRF-2019R1C1C1003844.  P. Shin was supported by NRF-2015R1A4A1041675.}

%    General info

%\date{\today}
\date{October 29, 2021 and, in revised form, January 20, 2022}
%\date{June 28, 2016 and, in revised form, December 2, 2016}
%\dedicatory{This paper is dedicated to our authors.}


\begin{abstract}
We prove global gradient estimates for parabolic $p$-Laplace type equations with measure data, whose model is
$$u_t - \textrm{div}  \left(|Du|^{p-2} Du\right) = \mu \quad \textrm{in} \ \Omega \times (0,T) \subset \mathbb{R}^n \times \mathbb{R},$$
where $\mu$ is a signed Radon measure with finite total mass. We consider the singular case
$$\frac{2n}{n+1} <p \le 2-\frac{1}{n+1}$$
and give possibly minimal conditions on the nonlinearity and the boundary of $\Omega$, which guarantee the regularity results for such measure data problems.
\end{abstract}

\maketitle

%%%%%%%%%%%%%%%%%%%%%%%%%%%%%%%%%%%%%%%%%%%%%%%%%%%%%%%%%

%{
%  \hypersetup{linkcolor=black}
%  \tableofcontents
%}

%%%%%%%%%%%%%%%%%%%%%%%%%%%%%%%%%%%%%%%%%%%%%%%%%%%%%%%%%%%%%%%%%%%%%%%%%%%%%%%%%%%%%%%%%%%
\section{Introduction and results}\label{Introduction}

Partial differential equations with measure data allow to take into account a variety of models in the area of applied mathematics: for instance, the flow pattern of blood in the heart \cite{Pes77, PM89}, surface tension forces concentrated on the interfaces of fluids \cite{SSO94, LL94, OF03}, and state-constrained optimal control theory \cite{Cas86, Cas93, CdT08, MPS11}.

In this paper, we establish global gradient estimates for solutions of quasilinear parabolic equations with measure data, having the form
\begin{equation}
\label{pem1}\left\{
\begin{alignedat}{3}
u_t -\ddiv \mathbf{a}(Du,x,t)  &= \mu &&\quad \text{in}  \  \OO_T, \\
u  &= 0 &&\quad \text{on} \  \partial_p \OO_T.
\end{alignedat}\right.
\end{equation}
Here $\OO_T := \OO \times (0,T)$ is a cylindrical domain with parabolic boundary $\partial_p \OO_T := \gh{\partial \OO \times [0,T]} \cup \gh{\OO \times \{0\}}$, where $\Omega \subset \bb^n$ is a bounded domain with nonsmooth boundary $\partial \Omega$, $n \ge 2$ and $T>0$. The nonhomogeneous term $\mu$ is a signed Radon measure on $\OO_T$ with finite total mass. From now on we assume that the measure $\mu$ is defined on $\bb^{n+1}$ by letting zero outside $\OO_T$; that is,
\begin{equation*}
|\mu|(\OO_T) = |\mu|(\bb^{n+1})< \infty.
\end{equation*}
A typical model of the problem \eqref{pem1} is given by the parabolic $p$-Laplace equation; that is,
$$u_t - \ddiv\gh{|Du|^{p-2}Du}  = \mu.$$
Throughout the paper, the nonlinearity $\mathbf{a}=\mathbf{a}(\xi,x,t): \bb^n \times \bb^n \times \bb \rightarrow \bb^n$ is
measurable in $(x,t)$ and $C^1$-regular in $\xi$, satisfying
\begin{equation}\label{str1}\left\{
\begin{aligned}
& |\mathbf{a}(\xi,x,t)| + |\xi||D_{\xi}\mathbf{a}(\xi,x,t)| \le \Lambda_1 |\xi|^{p-1},\\
& \Lambda_0 |\xi|^{p-2}|\eta|^2 \le \left< D_{\xi}\mathbf{a}(\xi,x,t)\eta,\eta \right>
\end{aligned}\right.
\end{equation}
for almost every $(x, t) \in \mathbb{R}^n \times \bb$, for every $\eta \in \mathbb{R}^n$, $\xi \in \bb^n
\setminus \{0\}$ and for some constants $\Lambda_1 \ge \Lambda_0 >0$.
Note that \eqref{str1} implies $\mathbf{a}(0,x,t)=0$ for $(x,t) \in \bb^n \times \bb$ and the following monotonicity condition:
\begin{equation}\label{monotonicity}
\left< \mathbf{a}(\xi_1,x,t)-\mathbf{a}(\xi_2,x,t), \xi_1-\xi_2 \right> \ge \tilde{\La}_0 \left( \norm{\xi_1}^2 + \norm{\xi_2}^2  \right)^{\frac{p-2}{2}} \norm{\xi_1-\xi_2}^2
\end{equation}
for all $(x, t) \in \mathbb{R}^n \times \bb$ and $\xi_1, \xi_2 \in \bb^n$, and for some constant
$\tilde{\La}_0=\tilde{\La}_0(n,\La_0,p)>0$.
In this paper we shall focus on the singular case
\begin{equation}\label{p-range}
\frac{2n}{n+1} <p \le 2-\frac{1}{n+1},
\end{equation}
since the case $p>2-\frac{1}{n+1}$ has been treated in \cite{BPS21}. The lower bound \eqref{p-range} is sharp in the sense that it reflects the fundamental solution of the parabolic $p$-Laplace equation (see Section \ref{Renormalized solutions} below).
For more detailed information concerning the lower bound \eqref{p-range}, we refer to Remark \ref{main rmk1} below.

%%%%%%%%%%%%%
\subsection{Renormalized solutions}\label{Renormalized solutions}
Let us first consider the parabolic $p$-Laplace equation with Dirac measure
\begin{equation*}\label{baren}
u_t - \ddiv \gh{|Du|^{p-2} Du} = \delta_0 \quad \text{in} \hspace{2mm} \bb^n \times \bb,
\end{equation*}
where $p\neq2$ and $\delta_0$ is the Dirac measure charging the origin. Then the {\em fundamental solution} $\Gamma$ is given by
\begin{equation}\label{baren sol}
\Gamma(x,t) = \left\{
\begin{alignedat}{3}
&t^{-n\theta} \gh{c(n,p) - \frac{p-2}{p} \theta^{\frac{1}{p-1}} \gh{\frac{|x|}{t^\theta}}^{\frac{p}{p-1}} }_+^{\frac{p-1}{p-2}}&&\quad \text{if} \ \ t>0,\\
&0&&\quad \text{otherwise},
\end{alignedat}\right.
\end{equation}
where $\theta := \frac{1}{p(n+1) - 2n}$. The solution $\Gamma$ is well defined provided
$$\theta >0 \iff p > \frac{2n}{n+1}.$$
We can check by a direct calculation that
$$D\Gamma \in L^q (\bb^{n+1}) \quad \text{for all} \ \ q < p - \frac{n}{n+1},$$
which implies that the solution $\Gamma$ does not belong to the usual energy space. Furthermore, we emphasize $D\Gamma \not\in L^1 (\bb^{n+1})$ if $p\le 2-\frac{1}{n+1}$. Under \eqref{p-range}, we thus need a proper notion of generalized solution as well as its gradient. For this, let us introduce a nonlinear parabolic capacity.
For every $p>1$ and every open subset $Q \subset \OO_T$, the {\em $p$-parabolic capacity} of $Q$ is defined by
$$\capp(Q) := \inf\mgh{\|u\|_W : u \in W,   u \ge \chi_Q \ \text{a.e. in} \ \OO_T},$$
where $\chi_Q$ is the usual characteristic function of $Q$ and
$$W := \mgh{u \in L^p (0,T; V) : u_t \in L^{p'}(0,T; V')}$$
endowed with the norm
$$\|u\|_W := \|u\|_{L^p (0,T; V)} + \|u_t\|_{L^{p'}(0,T; V')}.$$
Here $p'$ is the H\"{o}lder conjugate of $p$ with $\frac1p + \frac{1}{p'} =1$,
$V:=W_0^{1,p}(\OO) \cap L^2(\OO)$ and $V'$ is the dual space of $V$. For $p>\frac{2n}{n+2}$, the embedding $W_0^{1,p}(\OO) \subset L^2(\OO)$ is valid, so that $V:=W_0^{1,p}(\OO)$. We say a function $u$ is {\em $\capp$-quasi continuous} if for each $\ep >0$, there exists an open set $\tilde{Q} \subset \OO_T$ such that $\capp(\tilde{Q})<\ep$ and $u$ is continuous on $\OO_T \setminus \tilde{Q}$. Note that every function in $W$ has a $\capp$-quasi continuous representative.
We refer to \cite{Pie83,DPP03,KKKP13,AKP15,Ngu_pre} for further information concerning parabolic capacities.

Let $\mathfrak{M}_b(\OO_T)$ be the space of all signed Radon measures on $\OO_T$ with finite total mass. We denote by $\mathfrak{M}_a(\OO_T)$ the subspace of $\mathfrak{M}_b(\OO_T)$, which is absolutely continuous with respect to the $p$-parabolic capacity. We also denote by $\mathfrak{M}_s(\OO_T)$ the space of finite signed Radon measures in $\OO_T$ with support on a set of zero $p$-parabolic capacity. Then a measure $\mu \in \mathfrak{M}_b(\OO_T)$ can be uniquely decomposed into the following two components: (see \cite[Lemma 2.1]{FST91})
$$\mu = \mu_a +\mu_s, \quad \mu_a \in \mathfrak{M}_a(\OO_T), \quad \mu_s \in \mathfrak{M}_s(\OO_T).$$
Also, $\mu_a \in \mathfrak{M}_a(\OO_T)$ if and only if $\mu_a$ can be written as sum of the following functions:
$$\mu_a = f + g_t + \ddiv G,$$
where $f \in L^1(\OO_T)$, $g \in L^p (0,T; V)$ and $G \in L^{p'}(\OO_T)$ (see \cite{DPP03, KR19}).
We write $\mu = \mu^+ - \mu^-$, where $\mu^+$ and $\mu^-$ are the positive and negative parts, respectively, of a measure $\mu \in \mathfrak{M}_b(\OO_T)$ and set $|\mu| := \mu^+ + \mu^-$.
%We can write $\mu = \mu_0^+ - \mu_0^- + \mu_s^+ - \mu_s^-$, where $\mu_0^+$ (or $\mu_s^+$) and $\mu_0^-$ (or $\mu_s^-$) are the positive and negative parts, respectively, of a measure $\mu_a \in \mathfrak{M}_a(\OO_T)$ (or $\mu_s \in \mathfrak{M}_s(\OO_T)$) and set $|\mu| := \mu_0^+ + \mu_0^- + \mu_s^+ + \mu_s^-$.

Let us define the truncation operator
\begin{equation}\label{T_k}
T_k(s) := \max\mgh{-k,\min\mgh{k,s}} \quad \text{for any} \ k>0 \ \text{and} \ s \in \bb.
\end{equation}
If $u$ is a measurable function defined in $\OO_T$, finite almost everywhere, such that $T_k(u) \in \lpwwz$ for any $k>0$, then there exists a unique measurable function $U$ such that $DT_k(u) = U\chi_{\{|u|<k\}}$ a.e. in $\OO_T$ for all $k>0$.  In this case, we denote the spatial gradient $Du$ of $u$ by $Du:=U$. If $u \in L^1(0,T;W_0^{1,1}(\OO))$, then it coincides with the usual weak gradient.

Now we introduce the definition of renormalized solution given in \cite{PP15}.
\begin{definition}\label{renormalsol}
Let $p>1$ and let $\mu = \mu_a + \mu_s \in \mathfrak{M}_b(\OO_T)$ with $\mu_a \in \mathfrak{M}_a(\OO_T)$ and $\mu_s \in \mathfrak{M}_s(\OO_T)$. A function $u \in L^1(\OO_T)$ is a {\em renormalized solution} of the problem \eqref{pem1} if $T_k(u) \in \lpwwz$ for every $k>0$ and the following property holds: for any $k>0$ there exist sequences of nonnegative measures $\nu_k^+, \nu_k^- \in \mathfrak{M}_a(\OO_T)$ such that
\begin{equation*}
\nu_k^+ \to \mu_s^+,  \ \nu_k^- \to \mu_s^- \quad \text{tightly as} \ k\to \infty
\end{equation*}
and
\begin{equation}\label{tweaksol}
-\integral{\OO_T}{T_k(u)\varphi_t}{dxdt} + \integral{\OO_T}{\vgh{\mathbf{a}(DT_k(u),x,t), D\varphi}}{dxdt}
= \integral{\OO_T}{\varphi}{d\mu_k}
\end{equation}
for every $\varphi \in W \cap L^{\infty}(\OO_T)$ with $\varphi(\cdot,T)=0$, where $\mu_k := \mu_a + \nu_k^+ - \nu_k^-$.
\end{definition}
Here we say that a sequence $\{\mu_k\} \subset \mathfrak{M}_b(\OO_T)$ converges {\em tightly} (or {\em in the narrow topology of measures}) to $\mu \in \mathfrak{M}_b(\OO_T)$ if
$$\lim_{k\to\infty} \integral{\OO_T}{\varphi}{d\mu_k} =\integral{\OO_T}{\varphi}{d\mu}$$
for every bounded and continuous function $\varphi$ on $\OO_T$.

\begin{remark}\label{renormalsol rmk1}
Since $\varphi\in W$, a test function $\varphi$ admits a unique $\capp$-quasi continuous representative. This and the regularity of $T_k(u)$ imply that every term of \eqref{tweaksol} is well defined. Furthermore, \eqref{tweaksol} is equivalent to
\begin{equation}\label{tweaksol2}
T_k(u)_t -\ddiv \mathbf{a}(DT_k(u),x,t)  = \mu_k \quad \text{in}  \  \mathcal{D}'(\OO_T).
\end{equation}
Since $T_k(u)$ belongs to $\lpwwz$, we observe from \eqref{str1} that the measure $\mu_k$ belongs to $L^{p'}(0,T; W^{-1,p'}(\OO))$; hence we can regard $T_k(u)$ as a kind of weak solution to \eqref{tweaksol2}. This is necessary as we select a test function to obtain some comparison estimates in Section \ref{intrinsic comparison} below.
\end{remark}

\begin{remark}\label{renormalsol rmk2}
A renormalized solution $u$ of \eqref{pem1} becomes a distributional solution; this is, $u$ satisfies
\begin{equation*}
-\integral{\OO_T}{u\varphi_t}{dxdt} + \integral{\OO_T}{\vgh{\mathbf{a}(Du,x,t), D\varphi}}{dxdt} = \integral{\OO_T}{\varphi}{d\mu}
\end{equation*}
for any $\varphi \in C_c^\infty(\OO_T)$ (see \cite[Proposition 3]{PP15}).  Moreover, we note that if $\mu \in L^{p'}(0,T; W^{-1,p'}(\OO))$, then a renormalized solution coincides with a weak solution (see \cite{PPP11,PP15}).
\end{remark}

The notion of {\em renormalized solution} was first introduced by DiPerna and Lions \cite{DL89a, DL89b} for study of the Boltzmann and transport equations. This notion is adapted to obtain existence results for elliptic $p$-Laplace type equations with general measure data ($\mu \in \mathfrak{M}_b(\OO)$) by Dal Maso, Murat, Orsina and Prignet \cite{DMOP99}. For parabolic $p$-Laplace type problems, we refer to \cite{BM97} for the case of $L^1$ data ($\mu \in L^1(\OO_T)$) and \cite{DPP03,PPP11} for the case of diffuse (or soft) measure data ($\mu \in \mathfrak{M}_a(\OO_T)$). For the case with general measure data ($\mu \in \mathfrak{M}_b(\OO_T)$), Petitta \cite{Pet08} proved the existence of a renormalized solution when $p>2-\frac{1}{n+1}$ (see \cite{BVN15} for stability results), and Petitta and Porretta \cite{PP15} later generalize the result for $p>1$. It is also worthwhile to note that there are different notions of solutions for elliptic and parabolic measure data problems: {\em SOLA} (Solution Obtained by Limits of Approximations, see  \cite{BG89,BG92,Dal96,BDGO97}), {\em entropy solution} (see \cite{BBGGPV95,BGO96,Pri97}), and {\em superparabolic solution} (see \cite{KLP10,KLP13}). On the other hand, the uniqueness of a renormalized solution for parabolic measure data problems such as \eqref{pem1} remains a major open problem except the following special cases: (i) $\mu \in L^1(\OO_T)$ (see \cite{BM97}), (ii) $\mu \in \mathfrak{M}_a(\OO_T)$ (see \cite{DPP03,PPP11}), or (iii) the linear case; that is, $\mathbf{a}(\xi,x,t) =\mathbf{a}(x,t)\xi$ (see \cite[Section 9]{Pet08}).

%%%%%%%%%%%%%%%%%%%%%%%%%%%%%%%%%%%%%%%%%
\subsection{Main results}\label{Main theorems}
The aim of this paper is to establish global gradient estimates for renormalized solutions to the problem \eqref{pem1}. For this, let us first introduce the regularity assumptions on the nonlinearity $\mathbf a$ and the boundary of $\Omega$ (see Section \ref{Preliminaries} below for our basic notation).

\begin{definition}\label{main assumption}
Let $R>0$ and $\delta \in \gh{0,\frac{1}{8}}$. We say $(\mathbf{a},\OO)$ is {\em $(\delta,R)$-vanishing} if
\begin{enumerate}[(i)]
\item\label{2-small BMO} the nonlinearity $\mathbf{a}(\xi,x,t)$ satisfies
\begin{equation}\label{2-small BMO2}
\sup_{t_1,t_2 \in \bb} \sup_{0<r \le R} \sup_{y\in\bb^n} \fint_{t_1}^{t_2} \mint{B_r(y)}{\Theta\gh{\mathbf{a},B_r(y)}(x,t)}{dx dt} \le \delta,
\end{equation}
where
\begin{equation*}\label{2-AA}
\Theta\gh{\mathbf{a},B_r(y)}(x,t) := \sup_{\xi \in \bb^n \setminus \{0\}} \frac{\left|\mathbf{a}(\xi,x,t) - \mint{B_r(y)}{\mathbf{a}(\xi,\tilde{x},t)}{d\tilde{x}}  \right|}{|\xi|^{p-1}};
\end{equation*}

\item for each $y_0 \in \partial\OO$ and each $r \in (0,R]$, there exists
a new coordinate system $\{y_1,\cdots,y_n\}$ such that in this
coordinate system, the origin is $y_0$ and
\begin{equation*}\label{Reifenberg}
 B_r(0) \cap \mgh{y\in\bb^n : y_n > \delta r} \subset B_r(0) \cap \OO \subset B_r(0) \cap \mgh{y\in\bb^n : y_n > - \delta r}.
\end{equation*}
\end{enumerate}
\end{definition}

\begin{remark}\label{main remark}
\begin{enumerate}[(i)]
%\item This number is invariant under the scaling of the problem (\ref{pem1}), and the number $R$ is given arbitrary.

\item The first assumption of Definition \ref{main assumption} implies that the map $x \mapsto \frac{\mathbf{a}(\xi,x,t)}{|\xi|^{p-1}}$ is of {\em BMO (Bounded Mean Oscillation)} such that its BMO seminorm is less than or equal to $\delta$, uniformly in $\xi$ and $t$.

\item If the second condition of Definition \ref{main assumption} holds, then we say $\OO$ is called a {\em $(\delta,R)$-Reifenberg flat domain}. This domain includes Lipschitz domain with a sufficiently small Lipschitz constant and has the following geometric properties:
\begin{equation}\label{measure density}\left\{
\begin{aligned}
& \sup_{0<r \le R} \sup_{y \in \OO} \frac{|B_r(y)|}{|\OO \cap B_r(y)|} \le \gh{\frac{2}{1-\delta}}^n \le \gh{\frac{16}{7}}^n,\\
& \inf_{0<r \le R} \inf_{y \in \partial\OO} \frac{|\OO^c \cap B_r(y)|}{|B_r(y)|} \ge \gh{\frac{1-\delta}{2}}^n \ge \gh{\frac{7}{16}}^n.
\end{aligned}\right.
\end{equation}
For a further discussion on Reifenberg flat domains, see \cite{BW04,Tor97,LMS14} and the references therein.
\end{enumerate}
\end{remark}

We are ready to state the first main result of this paper.
\begin{theorem}\label{main theorem}
Let $\frac{2n}{n+1}<p\le2-\frac{1}{n+1}$ and let $0 < q <\infty$. Then there exists a small constant $\delta=\delta(n,\La_0,\La_1,p,q) >0$ such that the following holds: if $(\mathbf{a},\OO)$ is $(\delta,R)$-vanishing for some $R>0$, then for any renormalized solution $u$ of the problem \eqref{pem1} we have
\begin{equation}\label{main r1-2}
\integral{\OO_T}{|Du|^q}{dx dt} \le c \mgh{ \int_{\OO_T} \bgh{\M_1(\mu)}^{\frac{2q}{(n+1)p-2n}} dx dt  + \bgh{|\mu|(\OO_T)}^{\beta_0}}
\end{equation}
for some constant $c = c(n,\La_0,\La_1,p,q,R,\OO_T) \ge 1$, where $\beta_0:=\min\mgh{1,\frac{q(n+2)}{(n+1)p-n}}$.
\end{theorem}

Here  the {\em fractional maximal function of order $1$ for $\mu$}, denoted by $\M_1(\mu)$, is defined as
\begin{equation}\label{fractional mf}
\m_1(\mu) (x,t) := \sup_{r>0} \frac{ |\mu|(Q_r(x,t))}{r^{n+1}} \quad \text{for} \ \ (x,t) \in \bb^{n} \times \bb.
\end{equation}
Note that
\begin{equation}\label{main-p2}
\bgh{|\mu|(\OO_T)}^{\alpha} \leq c(n, \OO_T) \integral{\OO_T}{\bgh{\m_1 (\mu)}^{\alpha}}{dxdt} \quad \text{for all} \ \alpha >0.
\end{equation}

\begin{remark}\label{main rmk1}
\begin{enumerate}[(i)]
\item\label{anisotropic structure}  The lower bound of $p$ in Theorem \ref{main theorem} comes from the two followings: (1)  comparison estimates below $L^1$ spaces (see Remark \ref{comparmk} later); (2) the exponent $\frac{2}{(n+1)p-2n}>0$ in \eqref{main r1-2}.

\item We note that both the constant $c$ and the exponent $\frac{2}{(n+1)p-2n}$ in \eqref{main r1-2} tend to $+\infty$ as $p \searrow \frac{2n}{n+1}$. The exponent $\frac{2}{(n+1)p-2n}$ reflects the anisotropic structure (a constant multiple of a solution no longer yields another solution) of the problem \eqref{pem1} as well as the structure of the fundamental solution \eqref{baren sol}. Specifically, this exponent comes from a geometric difference between standard and intrinsic parabolic cylinders (see \cite[Lemma 5.4]{BPS21}). The exponent $\frac{2}{(n+1)p-2n}$ also appears in parabolic potential estimates (see \cite{KM13b,DZ21b}).

\item If $\frac{2n}{n+1} < p \le2$, then we have $\frac{n+2}{(n+1)p-n} \le \frac{2}{(n+1)p-2n}$. From \eqref{main r1-2} and \eqref{main-p2}, we deduce
\begin{equation*}
\integral{\OO_T}{|Du|^q}{dx dt} \le c \mgh{ \int_{\OO_T} \bgh{\M_1(\mu)}^{\frac{2q}{(n+1)p-2n}} dx dt  + 1}.
\end{equation*}
\end{enumerate}
\end{remark}

\begin{remark}\label{main rmk2}
We remark that the elliptic counterpart of the estimate \eqref{main r1-2} (under the range $1 < p \le 2-\frac{1}{n}$) was proved by Nguyen and Phuc \cite{NP19,NP20b}. Also, they have recently obtained pointwise potential estimates for the same elliptic problems under the range $\frac{3n-2}{2n-1} < p \le 2-\frac{1}{n}$ (see \cite{NP20}). On the other hand, we refer to \cite{Min07,DM10,DM11,Min10,Phu14a,Phu14b,Phu14c,KM18,Sch12a,Sch12b} for various regularity results for elliptic measure data problems with $p>2-\frac{1}{n}$.
\end{remark}

If the measure $\mu$ is time-independent or can be decomposed as in \eqref{measure decomposition} below, then we derive more sharp gradient estimate than the estimate \eqref{main r1-2}:
\begin{theorem}\label{main theorem2}
Let $\frac{2n}{n+1}<p\le2-\frac{1}{n+1}$ and let $p-1 < q <\infty$. Suppose that the following decomposition holds:
\begin{equation}\label{measure decomposition}
\mu=\mu_0 \otimes f,
\end{equation}
where $\mu_0$ is a finite signed Radon measure on $\OO$ and $f \in L^{\frac{q}{p-1}}(0,T)$. Then there exists a small constant $\delta=\delta(n,\La_0,\La_1,p,q) >0$ such that the following holds: if $(\mathbf{a},\OO)$ is $(\delta,R)$-vanishing for some $R>0$, then for any renormalized solution $u$ of the problem \eqref{pem1} we have
\begin{equation}\label{main r2-2}
\integral{\OO_T}{|Du|^q}{dx dt} \le c \mgh{\integral{\OO_T}{\bgh{\left(\M_1(\mu_0)\right)f}^{\frac{q}{p-1}}}{dx dt} + \bgh{|\mu_0|(\OO)\|f\|_{L^1(0,T)}}^{\beta_0}}
\end{equation}
for some constant $c = c(n,\La_0,\La_1,p,q,R,\OO_T) \ge 1$, where $\beta_0:=\min\mgh{1,\frac{q(n+2)}{(n+1)p-n}}$.
\end{theorem}
Here the (elliptic) fractional maximal function $\M_1(\mu_0)$ is given by
\begin{equation}\label{fractional mf2}
\m_1(\mu_0) (x) := \sup_{r>0} \frac{ |\mu_0|(B_r(x))}{r^{n-1}} \quad \text{for} \ \ x \in \bb^n.
\end{equation}
Note that
\begin{equation}\label{main-p2.6}
\bgh{|\mu_0|(\OO)}^{\alpha} \leq c(n, \OO) \integral{\OO}{\bgh{\m_1 (\mu_0)}^{\alpha}}{dx} \quad \text{for all} \ \alpha >0.
\end{equation}

\begin{remark}\label{main rk4}
\begin{enumerate}[(i)]
\item Unlike \eqref{main r1-2}, the estimate \eqref{main r2-2} has the form of elliptic estimates (cf. \cite{Min10,Phu14a,Phu14c,NP19,NP20b}).
Since $\frac{2}{(n+1)p-2n} > \frac{1}{p-1}$ under \eqref{p-range}, we observe that \eqref{main r2-2} gives a more sharp result.

\item If $1< p \le2$, then we have $\frac{n+2}{(n+1)p-n} \le \frac{1}{p-1}$. From \eqref{main r2-2} and \eqref{main-p2.6}, we deduce
\begin{equation*}
\integral{\OO_T}{|Du|^q}{dx dt} \le c \mgh{\integral{\OO_T}{\bgh{\left(\M_1(\mu_0)\right)f}^{\frac{q}{p-1}}}{dx dt} + 1}.
\end{equation*}
%\begin{enumerate}[(i)]
%\item The exponent $\frac{n+2}{(n+1)p-n}$ in \eqref{main r2-2} comes from the standard energy type estimates for \eqref{pem1} (see Lemma \ref{standard estimate} later). As noted in \eqref{anisotropic structure} of Remark \ref{main rmk1}, this exponent is also related to the anisotropic structure of\eqref{pem1}.

%\item Note that $\frac{2}{(n+1)p-2n} > \max\mgh{\frac{1}{p-1}, \frac{(n+2)}{(n+1)p-n}}$ under \eqref{p-range}. Thus, comparing the two estimates \eqref{main r1-2} and \eqref{main r2-2}, we observe that \eqref{main r2-2} gives a more sharp result.
\end{enumerate}
\end{remark}

%%%%%%%%%%%%%%%%%%%
\subsection{Novelty and outline of the paper}\label{novelty}

There have been many regularity results for parabolic measure data problems: for instance, Calder\'{o}n-Zygmund type estimates (see \cite{BPS21}), potential estimates (see \cite{KM14b,KM14c,KM13b}), and Marcinkiewicz estimates (see \cite{Bar14,Bar17,BD18}).
These results are based on the fact that the spatial gradient of a solution belongs to at least the $L^1$ space, thereby the assumption $p>2-\frac{1}{n+1}$ is essential to obtain such regularity estimates.
However, as mentioned earlier in Section \ref{Renormalized solutions}, the fundamental solution \eqref{baren sol} is indeed valid when $p>\frac{2n}{n+1}$.

The aim of the present paper is to fill this gap by developing the global gradient estimates for the problem \eqref{pem1} under $\frac{2n}{n+1}<p \le 2-\frac{1}{n+1}$ (see Theorems \ref{main theorem} and \ref{main theorem2}).
The main difficulty in obtaining Theorems \ref{main theorem} and \ref{main theorem2} lies in that the spatial gradient of a renormalized solution $u$ of \eqref{pem1} could not belong to the $L^1$ space.
To overcome this situation, we construct some comparison estimate below $L^1$. More precisely, we will show that if $w$ is a weak solution of the homogeneous problem $w_t -\ddiv \mathbf{a}(Dw,x,t) = 0$, then $|Du-Dw|$ is bounded in $L^\theta$, the constant $\theta \in (0,1)$ to be determined later, under in particular the range $\frac{3n+2}{2n+2}< p \le 2- \frac{1}{n+1}$ (see Lemma \ref{compa1} below).

This paper is structured as follows:
\begin{itemize}
\item In Section \ref{Preliminaries}, we collect basic notation and preliminary results used throughout the paper.

\item Section \ref{intrinsic comparison} commences with the $L^\theta$-comparison estimate between $Du$ and $Dw$ (Lemma \ref{compa1}). We investigate as well a higher integrability for $Dw$ (Lemma \ref{high int}) and regularity results such as Lipschitz continuity for reference problems, with the purpose of obtaining local comparison estimates below $L^1$ (Propositions \ref{comparison} and \ref{comparison2}).

\item\label{sec4} In Section \ref{la covering arguments}, we derive decay estimates for the spatial gradient of a renormalized solution $u$ (Proposition \ref{covarg}). For this, we apply a covering argument (Lemma \ref{VitCov}) developed in \cite{BPS21,BPS20} under intrinsic parabolic cylinders, having the intrinsic (fractional) maximal operators. Then we describe a relationship between intrinsic and standard fractional maximal functions (Lemmas \ref{difference of ratio} and \ref{difference of ratio1.5}).

\item In Section \ref{Global gradient estimates for parabolic measure data problems}, using the decay estimates obtained in Section \ref{la covering arguments}, we finally prove Theorems \ref{main theorem} and \ref{main theorem2}.
\end{itemize}

%%%%%%%%%%%%%%%%%%%%%%%%%%%%%%%%%%%%%%%%%%%%%%%%%%%%%%
\section{Preliminaries}\label{Preliminaries}

Let us first introduce basic notation, which will be used later. We denote by $c$ to mean a universal positive constant that can be computed in terms of known quantities; the exact value denoted by $c$ may be different from line to line.
A point $x \in \bb^n$ will be written $x = (x_1, \cdots, x_n)$. Let $B_r(x_0)$ denote the open ball in $\bb^n$ with center $x_0$ and radius $r>0$, and let $B_r^+(x_0) := B_r(x_0) \cap \{x \in \bb^n : x_n >0 \}$. We denote by
$$Q_r(x_0,t_0):=B_r(x_0) \times (t_0-r^2,t_0+r^2)$$
the {\em standard parabolic cylinder} in $\bb^n \times \bb =: \bb^{n+1}$ with center $(x_0,t_0) \in \bb^{n+1}$, radius $r$ and height $2r^2$. With $\lambda > 0$, we also consider the {\em intrinsic parabolic cylinder}
\begin{equation*}\label{ipc}
\iq{r}(x_0,t_0) := B_r(x_0) \times (t_0- \lambda^{2-p} r^2, t_0+ \lambda^{2-p} r^2),
\end{equation*}
see \cite{DiB93,Urb08,KM14b} for more detailed information of intrinsic geometry related to the intrinsic parabolic cylinder.
We also use the following notation:
\begin{gather*}
\OO_T := \OO \times (0,T), \ \omt := \OO \times (-\infty,T), \ \omtb := \OO \times (-T,T),\\
\ik{r}(x_0,t_0) := \iq{r}(x_0,t_0) \cap \omt, \ I_r^\la(t_0) :=(t_0- \lambda^{2-p} r^2, t_0+ \lambda^{2-p} r^2),\\
%\OO_r(x_0) := B_r(x_0) \cap \OO, \ K_r(x_0,t_0) := Q_r(x_0,t_0) \cap \OO_T, \ \ik{r}(x_0,t_0) := \iq{r}(x_0,t_0) \cap \OO_T,\\
\iqq{r}(x_0,t_0) := B_r^+(x_0) \times I_r^\la(t_0),\\
T_r^\lambda(x_0,t_0) := \gh{B_r(x_0) \cap \{x \in \bb^n : x_n =0 \}} \times I_r^\la(t_0).
\end{gather*}
%and for simplicity if $(x_0,t_0):=(0,0)$ we write $B_r := B_r(0)$, $Q_r := Q_r(0,0)$, $\iq{r} := \iq{r}(0,0)$, $\ik{r} := \ik{r}(0,0)$, and so on.
Let us use both the notation $f_t$ and $\partial_t f$ to denote the time derivative of a function $f$.
We denote by $Df$ the spatial gradient of $f$.
Given a real-valued function $f$, we write
\begin{equation*}
(f)_+ := \max\mgh{f,0} \quad \text{and} \quad (f)_- := - \min\mgh{f,0}.
\end{equation*}
For each set $Q \subset \bb^{n+1}$, $|Q|$ is the $(n+1)$-dimensional Lebesgue measure of $Q$ and $\chi_Q$ is the usual characteristic function of $Q$.
For $f \in L_{loc}^1(\bb^{n+1})$, $\bar{f}_{Q}$ stands for the integral average of $f$ over a bounded open set $Q \subset \bb^{n+1}$; that is,
$$\bar{f}_{Q} := \mint{Q}{f(x,t)}{dx dt} := \frac{1}{|Q|} \integral{Q}{f(x,t)}{dx dt}.$$

For $f \in  L_{loc}^1(\bb^{n+1})$ and $\lambda>0$, we define the {\em (intrinsic) $\lambda$-maximal function} of $f$ as
\begin{equation*}
\m^\lambda f (x,t) := \sup_{r>0} \mint{Q^\lambda_r(x,t)}{|f(y,s)|}{dyds}.
\end{equation*}
We write
\begin{equation}\label{intrinsic mf}
\m^\lambda_Q f := \m^\lambda \left(f \chi_Q\right)
\end{equation}
provided $f$ is defined on a set $Q \subset \bb^{n+1}$. In particular, it coincides the {\em classical maximal function} $\M f$ when $\lambda=1$ or $p=2$.

We give weak $(1,1)$-estimates for the $\la$-maximal function as follows:
\begin{lemma}\label{WE}
Let $Q$ be an open set in $\bb^{n+1}$. If $f \in L^1(Q)$, then there exists a constant $c=c(n)\ge1$ such that
\begin{align}\label{lambda 1-1}
\left|\mgh{(y,s)\in Q: \m_Q^\lambda f (y,s) >\alpha}\right| \leq \frac{c}{\alpha}\integral{Q}{|f(x,t)|}{dx dt}
\end{align}
for any $\alpha>0$. Moreover, we have
\begin{align*}\label{lambda 1-1_2}
\left|\mgh{(y,s)\in Q: \m_Q^\lambda f (y,s) > 2\alpha}\right| \leq \frac{c}{\alpha}\int_{\mgh{(y,s)\in Q: |f|>\alpha}} |f|\ dxdt
\end{align*}
for any $\alpha>0$.
\end{lemma}

\begin{proof}
The proof is directly obtained from \cite[Lemma 2.12]{BPS21} with $f$ replaced by $f \chi_Q$.
\end{proof}

We introduce a useful integral property, which can be easily computed by the Fubini theorem.
\begin{lemma}\label{useful int}
Let $Q$ be an open set in $\bb^{n+1}$. For any $q > l \ge 0$, we have
\begin{equation}\label{useful int1}
\integral{Q}{T_k(|f|)^{q-l} |f|^l}{dxdt} = (q-l) \int_0^k \la^{q-l-1}\bgh{\integral{Q\cap\mgh{|f| > \la}}{|f|^l}{dxdt}} d\la,
\end{equation}
where $T_k$ is the truncation operator defined in \eqref{T_k}. Furthermore, if $f \in L^{q}(Q)$, then \eqref{useful int1} also holds for $k=\infty$.
\end{lemma}

We also record an embedding theorem for parabolic Sobolev spaces.
\begin{lemma}[See {\cite[Chapter I, Proposition 3.1]{DiB93}}]
\label{parabolic embedding}
Let $q, l \ge1$ and let $Q:=B\times(t_1,t_2) \subset \bb^{n} \times \bb$. Then there is a constant $c=c(n,q,l)\ge1$ such that for every $f \in L^\infty(t_1,t_2;L^l(B)) \cap L^q(t_1,t_2;W_0^{1,q}(B))$, we have
\begin{equation*}
\integral{Q}{|f|^{q\frac{n+l}{n}}}{dxdt} \le c\gh{\integral{Q}{|Df|^q}{dxdt}}\gh{\sup_{t_1<t<t_2}\integral{B\times\{t\}}{|f|^l}{dx}}^{\frac{q}{n}}.
\end{equation*}

\end{lemma}

%%%%%%%%%%%%%%%%%%%%%%%%%%%%%%%%%%%%%%%%%%%%%%%%%%%%%%%%%%%%%%%%%%%%%%%%%%%%%%%%%%%%%

\section{Local comparison estimates below $L^1$ spaces}\label{intrinsic comparison}

In this section we derive local comparison estimates for the spatial gradient of a solution $T_k(u) \in \lpwwz$ to \eqref{tweaksol} in an intrinsic parabolic cylinder. Here we only consider comparison results near a boundary region, since the counterparts in an interior region can be done in the same way. Also, we obtain these estimates below $L^1$ spaces (see Lemma \ref{compa1} and Propositions \ref{comparison} and \ref{comparison2} below), as the spatial gradient of a renormalized solution $u$ to the problem \eqref{pem1} does not generally belong to $L^1(\OO_T)$ under the assumption $p\le 2-\frac{1}{n+1}$ (see Section \ref{Renormalized solutions}).
As noted in Remark \ref{renormalsol rmk1}, $T_k(u)$ becomes a weak solution of \eqref{tweaksol} with $\mu_k \in L^{p'}(0,T; W^{-1,p'}(\OO))$.
Throughout this section, we replace $T_k(u)$ by $u$ and $\mu_k$ by $\mu$, and we extend $u$ by zero for $t<0$ (see Remark \ref{vicrmk1} for the reason for this time extension).

Suppose that $(\mathbf{a},\OO)$ is $(\delta,R)$-vanishing for some $R>0$, where $\delta \in \gh{0,\frac18}$ is to be determined later. Fix any $\la >0$, $(x_0,t_0) \in \omt$ and $0 < r \le \frac{R}{8}$ satisfying
\begin{equation}\label{reifen}
B_{8r}^+(x_0) \subset B_{8r}(x_0) \cap \OO \subset B_{8r}(x_0) \cap \{x \in \bb^n : x_n > -16\delta r\}.
\end{equation}
In this section, we for simplicity omit denoting the center by $K_r^\la \equiv K_r^\la(x_0,t_0)$, $B_r \equiv B_r(x_0)$ and $I_r^\la \equiv I_r^\la(t_0)$.

Let $w$ be the unique weak solution to the Cauchy-Dirichlet problem
\begin{equation}
\label{pem2}\left\{
\begin{alignedat}{3}
w_t -\ddiv \mathbf{a}(Dw,x,t) &= 0 &&\quad \text{in} \ \ik{8r}, \\
w &= u &&\quad\text{on} \ \partial_p \ik{8r}.
\end{alignedat}\right.
\end{equation}

We first give a comparison estimate for the difference of $Du$ and $Dw$, which is a crucial estimate in this paper.
\begin{lemma}\label{compa1}
Let $\frac{3n+2}{2n+2}<p\le 2-\frac{1}{n+1}$, let $u$ be a weak solution of \eqref{tweaksol} and let $w$ as in \eqref{pem2} with \eqref{reifen}.
Then there exists a constant $ c=c(n,\La_0,p,\theta)\ge1$ such that
\begin{equation}\label{compa1-r}
\begin{aligned}
\gh{\mint{\ik{8r}}{|Du-Dw|^{\theta}}{dxdt}}^{\frac{1}{\theta}} &\le c \bgh{ \frac{|\mu|(\ik{8r})}{|\ik{8r}|^{\frac{n+1}{n+2}}} }^{\frac{n+2}{(n+1)p-n}}\\
&\quad + c \bgh{ \frac{|\mu|(\ik{8r})}{|\ik{8r}|^{\frac{n+1}{n+2}}} } \gh{\mint{\ik{8r}}{|Du|^{\theta}}{dxdt}}^{\frac{(2-p)(n+1)}{\theta(n+2)}}
\end{aligned}
\end{equation}
for any constant $\theta$ such that $\frac{n+2}{2(n+1)} < \theta < p - \frac{n}{n+1} \le 1$.
\end{lemma}

\begin{proof}
To streamline the proof, we will take the test functions in \eqref{st-test1} and \eqref{st-test2} below without the use of the so-called {\em Steklov average} (see \cite{DiB93} for its definition, properties and standard use).
For simplicity of notation, we temporarily write
\begin{equation*}
\begin{aligned}
\mathbf{a}(\xi) := \mathbf{a}(\xi,x,t) \ \text{and} \ \mathfrak{A}(a,b) :=  \mgh{(x,t) \in \ik{8r} : a<\gh{u-w}_\pm(x,t) <b^{\frac{1}{1-\gamma}}}
\end{aligned}
\end{equation*}
for any $0\le a<b^{\frac{1}{1-\gamma}}\leq \infty$ and $0\le \gamma <1$. Let us also introduce the vector field $V : \mr^n \to \mr^n$ defined by
$V(\xi) := |\xi|^\frac{p-2}{2}\xi$ for all $\xi \in \mr^n$. Note that for any $\xi_1, \xi_2\in \mr^n$, it holds
\begin{equation}\label{Vcomparable}
c^{-1} \left(|\xi_1|^2+|\xi_2|^2\right)^\frac{p-2}{2} \leq \frac{|V(\xi_1)-V(\xi_2)|^2}{|\xi_1-\xi_2|^2}\leq c \left(|\xi_1|^2+|\xi_2|^2\right)^\frac{p-2}{2}
\end{equation}
for some constant $c=c(n,p)\ge1$ (see \cite{Ham92,Min07} for a further discussion on the vector field $V$).

{\em Step 1.}
We will first show that
\begin{equation}\label{st-1}
\sup_{t \in I_{8r}^\la} \integral{\OO_{8r}}{|u-w|}{dx} \le |\mu|(\ik{8r})
\end{equation}
and
\begin{equation}\label{st-2}
\begin{aligned}
&\integral{\ik{8r}}{\frac{|u-w|^{-\gamma}  |V(Du)-V(Dw)|^2}{\gh{\alpha^{1-\gamma} + |u-w|^{1-\gamma}}^{\xi}}}{dx dt} \le c \frac{\alpha^{(1-\gamma)(1-\xi)}}{(1-\gamma)(\xi-1)} |\mu|(\ik{8r})
\end{aligned}
\end{equation}
for any $0\le \gamma <1$, $\alpha > 0$ and $\xi>1$, where $c = c(n,\La_0,p)\ge 1$ and $\OO_{8r} := B_{8r} \cap \OO$.  For any fixed $\ep$ and $\tilde{\ep}$ with $\varepsilon>\tilde{\ep}^{1-\gamma}>0$, choose a test function
\begin{equation}\label{st-test1}
\varphi_1 = \pm \min\mgh{1, \max\mgh{\frac{\gh{u-w}_\pm^{1-\gamma}-\tilde{\ep}^{1-\gamma}}{\ep-\tilde{\ep}^{1-\gamma}},0}}\zeta,
\end{equation}
where $\zeta : \bb \to [0,1]$ is a nonincreasing smooth function satisfying $\zeta(t)=0$ for all $t\ge\tau$ with $\tau \in I_{8r}^\la$. We then directly compute
\begin{equation*}
D\varphi_1 = \frac{1-\gamma}{\varepsilon-\tilde{\ep}^{1-\gamma}} \chi_{\mathfrak{A}(\tilde{\ep},\ep)} \zeta \gh{u-w}_\pm^{-\gamma} (Du-Dw).
\end{equation*}
Since $\gh{u-w}_\pm^{-\gamma} \le \tilde{\ep}^{-\gamma}$ on $\mathfrak{A}(\tilde{\ep},\ep)$, we have  $\varphi_1 \in L^p(I_{8r}^\la;W_0^{1,p}(\OO_{8r}))$ with $\left|\varphi_1 \right| \le 1$ and $\varphi_1(\cdot,\lambda^{2-p} (8r)^2)=0$.
Substituting $\varphi_1$ into the weak formulation of the subtracted equation of \eqref{tweaksol} and \eqref{pem2} and then integrating on $I_{8r}^\la$, we obtain
\begin{equation}\label{st-4}
\underbrace{\integral{\ik{8r}}{\partial_t (u-w) \varphi_1}{dx dt}}_{=: I_1} + \underbrace{\integral{\ik{8r}}{\vgh{\mathbf{a}(Du) - \mathbf{a}(Dw), D\varphi_1}}{dx dt}}_{=: I_2} = \underbrace{\integral{\ik{8r}}{\varphi_1}{d\mu}}_{=: I_3}.
\end{equation}
To estimate $I_1$, we compute
\begin{equation*}\label{st-5}
\begin{aligned}
&\partial_t (u-w) \min\mgh{1, \max\mgh{\frac{\gh{u-w}_\pm^{1-\gamma}-\tilde{\ep}^{1-\gamma}}{\ep-\tilde{\ep}^{1-\gamma}},0}}\\
&\qquad = \pm \partial_t \integraL{\tilde{\ep}}{\gh{u-w}_\pm}{\min\mgh{1, \max\mgh{\frac{s^{1-\gamma}-\tilde{\ep}^{1-\gamma}}{\ep-\tilde{\ep}^{1-\gamma}},0}}}{ds}.
\end{aligned}
\end{equation*}
Then the integration by parts gives
\begin{equation*}\label{st-6}
\begin{aligned}
I_1 &= \integral{\ik{8r}}{\bgh{\integraL{\tilde{\ep}}{\gh{u-w}_\pm}{\min\mgh{1, \frac{s^{1-\gamma}-\tilde{\ep}^{1-\gamma}}{\ep-\tilde{\ep}^{1-\gamma}}}}{ds}}  \gh{-\zeta_t}}{dx dt} \ge 0,
\end{aligned}
\end{equation*}
since $\zeta_t \le 0$. Also, we have from \eqref{monotonicity} that
\begin{gather*}
\label{st-7} I_2= \frac{1-\gamma}{\varepsilon-\tilde{\ep}^{1-\gamma}} \integral{\mathfrak{A}(\tilde{\ep},\ep)}{\zeta \gh{u-w}_\pm^{-\gamma} \vgh{\mathbf{a}(Du) - \mathbf{a}(Dw), Du-Dw}}{dx dt}  \ge 0 % \quad (\text{only if} \ \ s=0) ??????.
\end{gather*}
and from $\left|\varphi_1 \right| \le 1$ that
\begin{gather}
\label{st-8} |I_3| = \left|\integral{\ik{8r}}{\varphi_1}{d\mu}\right| \le |\mu|(\ik{8r}).
\end{gather}
Utilizing the three inequalities above in \eqref{st-4} and letting $\tilde{\ep} \to 0$, we derive
\begin{equation}\label{st-9}
\integral{\ik{8r}}{\bgh{\integraL{0}{\gh{u-w}_\pm}{\min\mgh{1, \frac{s^{1-\gamma}}{\ep}}}{ds}}  \gh{-\zeta_t}}{dx dt} \le |\mu|(\ik{8r})
\end{equation}
and
\begin{equation}\label{st-9.5}
\frac{1-\gamma}{\varepsilon} \integral{\mathfrak{A}(0,\ep)}{\zeta \gh{u-w}_\pm^{-\gamma} \vgh{\mathbf{a}(Du) - \mathbf{a}(Dw), Du-Dw} }{dx dt} \le |\mu|(\ik{8r}).
\end{equation}
As $\varepsilon \to 0$ in \eqref{st-9}, Lebesgue's dominated convergence theorem implies
\begin{equation*}\label{st-10}
\integral{\ik{8r}}{|u-w| \gh{-\zeta_t}}{dx dt} \le |\mu|(\ik{8r}).
\end{equation*}
We then let $\zeta$ approximate the characteristic function $\chi_{(-\infty,\tau)}$ to obtain
\begin{equation*}\label{st-10.5}
\integral{\OO_{8r}\times\{\tau\}}{|u-w|}{dx} \le |\mu|(\ik{8r})
\end{equation*}
for every $\tau \in I_{8r}^\la$, which implies the estimate \eqref{st-1}.

To obtain \eqref{st-2}, we take an another test function
\begin{equation}\label{st-test2}
\varphi_2 = \frac{\varphi_1}{\gh{\alpha^{1-\gamma}+\gh{u-w}_\pm^{1-\gamma}}^{\xi-1}},
\end{equation}
where $0\le \gamma <1$, $\alpha > 0$ and $\xi>1$ are to be determined later in a universal way. Testing $\varphi_2$ to the subtracted equation of \eqref{tweaksol} and \eqref{pem2} and then integrating over $I_{8r}^\la$, we get
\begin{equation}\label{st-13}
\integral{\ik{8r}}{\partial_t (u-w) \varphi_2}{dx dt} + \integral{\ik{8r}}{\vgh{\mathbf{a}(Du) - \mathbf{a}(Dw), D\varphi_2}}{dx dt} = \integral{\ik{8r}}{\varphi_2}{d\mu}.
\end{equation}
Since $\varphi_2 \leq \alpha^{(1-\gamma)(1-\xi)} \varphi_1$, we employ \eqref{st-8} and \eqref{st-9} to discover
\begin{gather*}
\label{st-15} \left|\lim_{\tilde{\ep}\to0} \integral{\ik{8r}}{\varphi_2}{d\mu}\right| \le \alpha^{(1-\gamma)(1-\xi)} |\mu|(\ik{8r})
\end{gather*}
and
\begin{gather*}
\label{st-14}  \lim_{\tilde{\ep}\to0} \integral{\ik{8r}}{\partial_t (u-w) \varphi_2}{dx dt} \le \alpha^{(1-\gamma)(1-\xi)} |\mu|(\ik{8r}).
\end{gather*}
%for all $\ep>0$.
To estimate the second term on the left-hand side of \eqref{st-13}, we establish
\begin{equation*}\label{st-16}
\begin{aligned}
& \integral{\ik{8r}}{\vgh{\mathbf{a}(Du) - \mathbf{a}(Dw), D\varphi_2}}{dx dt} \\
&\qquad = \integral{\ik{8r}}{\frac{\vgh{\mathbf{a}(Du) - \mathbf{a}(Dw), D\varphi_1}}{\gh{\alpha^{1-\gamma}+\gh{u-w}_\pm^{1-\gamma}}^{\xi-1}}}{dx dt}\\
&\qquad\qquad + (1-\xi) \integral{\ik{8r}}{ \frac{\varphi_1\vgh{\mathbf{a}(Du) - \mathbf{a}(Dw), D\gh{u-w}_\pm^{1-\gamma}}}{\gh{\alpha^{1-\gamma}+\gh{u-w}_\pm^{1-\gamma}}^{\xi}}}{dx dt}\\
&\qquad =: I_4 + I_5.
\end{aligned}
\end{equation*}
It follows from \eqref{st-9.5} that
\begin{equation*}\label{st-18}
\lim_{\tilde{\ep}\to0} I_4 \le \alpha^{(1-\gamma)(1-\xi)} |\mu|(\ik{8r}).
\end{equation*}
As $\tilde{\ep}\to0$, we have
\begin{equation*}\label{st-17}
\begin{aligned}
 I_5 \to (1-\xi) \integral{\ik{8r}}{\zeta \min\mgh{1, \frac{\gh{u-w}_\pm^{1-\gamma}}{\ep}}
\frac{\vgh{\mathbf{a}(Du) - \mathbf{a}(Dw), D\gh{u-w}^{1-\gamma}_\pm}}{\gh{\alpha^{1-\gamma}+\gh{u-w}_\pm^{1-\gamma}}^{\xi}}}{dx dt}.
\end{aligned}
\end{equation*}
We then insert the previous estimates into \eqref{st-13} and utilize \eqref{monotonicity} and \eqref{Vcomparable}, to discover
\begin{equation*}\label{st-19}
\begin{aligned}
&\integral{\ik{8r}}{\frac{|u-w|^{-\gamma} |V(Du)-V(Dw)|^2}{\gh{\alpha^{1-\gamma} + |u-w|^{1-\gamma}}^{\xi}}\min\mgh{1, \frac{|u-w|^{1-\gamma}}{\ep}}}{dx dt}\\
&\qquad \le c\frac{\alpha^{(1-\gamma)(1-\xi)}}{(1-\gamma)(\xi-1)} |\mu|(\ik{8r})
\end{aligned}
\end{equation*}
for some constant $c=c(n,\La_0,p)\ge1$. As $\ep\to0$, we obtain the estimate \eqref{st-2}.

\medskip

{\em Step 2.} Let $\theta$ be such that $\frac{n+2}{2(n+1)} < \theta < p - \frac{n}{n+1}\le1$. For fixed $\ep>0$, set $\mathfrak{B}_{\ep} := \mgh{(x,t) \in \ik{8r} : |u-w|>\ep}$.
Let $\beta  \in \left[0,\frac{p}{2}\right)$ be the constant satisfying $\frac{\beta}{p}= \frac{(1-\theta)(n+1)}{n}$ and define
$$
M_\ep :=\frac{p}{p-\beta}\mint{\ik{8r}}{\norm{D|u-w|^\frac{p-\beta}{p}}\chi_{\mathfrak{B}_{\ep}}}{dxdt}.
$$
Indeed, $M_\ep< \infty$ since $|u-w|>\ep$ on $\mathfrak{B}_{\ep}$.
We see from H\"older's inequality that
\begin{equation}\label{st2-1}
\begin{aligned}
&\mint{\ik{8r}}{\norm{Du-Dw}^{\theta}\chi_{\mathfrak{B}_{\ep}}}{dxdt} \\
&\qquad = \mint{\ik{8r}}{\gh{|u-w|^{-\frac{(1-\theta)(n+1)}{n}}\norm{Du-Dw}}^\theta|u-w|^\frac{(1-\theta)\theta(n+1)}{n}\chi_{\mathfrak{B}_{\ep}}}{dxdt} \\
&\qquad \leq M_\ep^\theta\gh{\mint{\ik{8r}}{|u-w|^\frac{\theta(n+1)}{n}\chi_{\mathfrak{B}_{\ep}}}{dxdt}}^{1-\theta}.
\end{aligned}
\end{equation}
Applying Lemma~\ref{parabolic embedding} with $f=\gh{|u-w|^\frac{p-\beta}{p}-\ep^\frac{p-\beta}{p}}_+$, $q=1$ and $l=\frac{p}{p-\beta}$, we find
\begin{equation}\label{st2-1.5}
\begin{aligned}
&\mint{\ik{8r}}{|u-w|^\frac{\theta (n+1)}{n}\chi_{\mathfrak{B}_{\ep}}}{dxdt} \\
&\quad\leq c \mint{\ik{8r}}{\gh{|u-w|^\frac{p-\beta}{p}-\ep^\frac{p-\beta}{p}}_+^\frac{\theta p(n+1)}{n(p-\beta)}}{dxdt} +c \ep^\frac{\theta (n+1)}{n}\\
&\quad\leq c M_\ep \gh{\sup_{t \in I_{8r}^\la} \integral{\OO_{8r}}{|u-w|}{dx}}^\frac{1}{n}+c \ep^\frac{\theta (n+1)}{n},
\end{aligned}
\end{equation}
by noting that $\frac{\theta (n+1)}{n}=\frac{(n+1)p-n\beta}{np}$. Let us set
\begin{equation}\label{def of alpha}
\alpha_\ep :=  \bgh{|\mu|(\ik{8r})M_\ep^n}^\frac{1}{\theta (n+1)}+\ep.
\end{equation}
Then it follows from \eqref{st-1} and \eqref{st2-1.5} that
\begin{equation}\label{st2-2}
\mint{\ik{8r}}{|u-w|^\frac{\theta (n+1)}{n}\chi_{\mathfrak{B}_{\ep}}}{dxdt} \leq c \alpha_\ep^\frac{\theta (n+1)}{n}.
\end{equation}
Inserting this inequality into \eqref{st2-1}, we obtain
\begin{equation}\label{st2-2.5}
\mint{\ik{8r}}{\norm{Du-Dw}^{\theta}\chi_{\mathfrak{B}_{\ep}}}{dxdt} \le c M_\ep^\theta \alpha_\ep^\frac{\theta(1-\theta) (n+1)}{n}
\end{equation}
for some constant $c(n,\La_0,p,\theta)\ge1$.

Now, we will estimate the quantity $M_\ep$. We notice from \eqref{Vcomparable} that
$$
|Du-Dw| \leq c \norm{V(Du)-V(Dw)}^\frac{2}{p}+ c |Du|^\frac{2-p}{2}\norm{V(Du)-V(Dw)}.
$$
Then we have
\begin{equation}\label{st2-3}
\begin{aligned}
M_\ep &\leq c \underbrace{\mint{\ik{8r}}{|u-w|^{-\frac{\beta}{p}}\norm{V(Du)-V(Dw)}^\frac{2}{p}\chi_{\mathfrak{B}_{\ep}}}{dxdt}}_{=: J_1}\\
& \qquad  + c \underbrace{\mint{\ik{8r}}{|u-w|^{-\frac{\beta}{p}}|Du|^\frac{2-p}{2}\norm{V(Du)-V(Dw)}\chi_{\mathfrak{B}_{\ep}}}{dxdt}}_{=: J_2}.
\end{aligned}
\end{equation}
From \eqref{st-2} with $\gamma=\beta \in [0,1)$ and $\alpha=\alpha_\ep>0$, we have
\begin{equation*}
\begin{aligned}
J_1&\leq \mint{\ik{8r}}{\gh{\frac{|u-w|^{-\beta}\norm{V(Du)-V(Dw)}^2 }{\gh{\alpha_\ep^{1-\beta} + |u-w|^{1-\beta}}^{\xi}}}^\frac{1}{p}\gh{\alpha_\ep^{1-\beta} + |u-w|^{1-\beta}}^\frac{\xi}{p}\chi_{\mathfrak{B}_{\ep}}}{dx dt} \\
&\leq \gh{\mint{\ik{8r}}{\frac{|u-w|^{-\beta} \norm{V(Du)-V(Dw)}^2 \chi_{\mathfrak{B}_{\ep}}}{\gh{\alpha_\ep^{1-\beta} + |u-w|^{1-\beta}}^{\xi}}}{dx dt}}^\frac{1}{p}\\
&\qquad \times  \gh{\mint{\ik{8r}}{\gh{\alpha_\ep^{1-\beta} + |u-w|^{1-\beta}}^\frac{\xi}{p-1}\chi_{\mathfrak{B}_{\ep}}}{dx dt}}^\frac{p-1}{p} \\
&\leq c \alpha_\ep^\frac{(1-\beta)(1-\xi)}{p}\bgh{\frac{|\mu|(\ik{8r})}{\norm{\ik{8r}}}}^\frac{1}{p}\mgh{\alpha_\ep^\frac{(1-\beta)\xi}{p}+\gh{\mint{\ik{8r}}{|u-w|^\frac{(1-\beta)\xi}{p-1}\chi_{\mathfrak{B}_{\ep}}}{dx dt}}^\frac{p-1}{p}}.
\end{aligned}
\end{equation*}
Furthermore since
$$
\theta < p - \frac{n}{n+1} \iff \frac{1-\beta}{p-1} < \frac{(n+1)p-n\beta}{np}=\frac{\theta (n+1)}{n},
$$
we can choose $\xi>1$ so that $\frac{(1-\beta)\xi}{p-1} < \frac{\theta (n+1)}{n}$,
and then we discover
\begin{equation*}
\begin{aligned}
\gh{\mint{\ik{8r}}{|u-w|^\frac{(1-\beta)\xi}{p-1}\chi_{\mathfrak{B}_{\ep}}}{dx dt}}^\frac{p-1}{p}
&\leq \gh{\mint{\ik{8r}}{|u-w|^\frac{\theta (n+1)}{n}\chi_{\mathfrak{B}_{\ep}}}{dxdt}}^\frac{(1-\beta)\xi n}{\theta p(n+1)}\\ &\leq c \alpha_\ep^\frac{(1-\beta)\xi}{p}
\end{aligned}
\end{equation*}
as a consequence of \eqref{st2-2}.
Therefore we have
\begin{equation}\label{st2-4.5}
J_1 \leq c \alpha_\ep^\frac{1-\beta}{p} \bgh{\frac{|\mu|(\ik{8r})}{\norm{\ik{8r}}}}^\frac{1}{p}.
\end{equation}
To estimate $J_2$, we have from \eqref{st-2} with $\gamma=\frac{2\beta}{p}  \in [0,1)$ and $\alpha=\alpha_\ep>0$ that
\begin{equation*}
\begin{aligned}
J_2&\leq \mint{\ik{8r}}{\gh{\frac{|u-w|^{-\gamma}\norm{V(Du)-V(Dw)}^2}{\gh{\alpha_\ep^{1-\gamma} + |u-w|^{1-\gamma}}^{\xi}}}^\frac{1}{2}\gh{\alpha_\ep^{1-\gamma} + |u-w|^{1-\gamma}}^\frac{\xi}{2}|Du|^\frac{2-p}{2}\chi_{\mathfrak{B}_{\ep}}}{dx dt} \\
&\leq c \alpha_\ep^\frac{(1-\gamma)(1-\xi)}{2}\bgh{\frac{|\mu|(\ik{8r})}{\norm{\ik{8r}}}}^\frac{1}{2}\gh{\mint{\ik{8r}}{\gh{\alpha_\ep^{1-\gamma} + |u-w|^{1-\gamma}}^\xi |Du|^{2-p}\chi_{\mathfrak{B}_{\ep}}}{dx dt}}^\frac{1}{2}.
\end{aligned}
\end{equation*}
From the fact that
$$
\theta > \frac{n+2}{2(n+1)} > 2-p \qquad \gh{\text{since}\ \ p>\frac{3n+2}{2n+2}},
$$
it follows
\begin{equation*}
\begin{aligned}
&\mint{\ik{8r}}{\gh{\alpha_\ep^{1-\gamma} + |u-w|^{1-\gamma}}^\xi|Du|^{2-p}\chi_{\mathfrak{B}_{\ep}}}{dx dt}\\
&\quad \leq \gh{\mint{\ik{8r}}{|Du|^\theta}{dx dt}}^\frac{2-p}{\theta}\gh{\mint{\ik{8r}}{\gh{\alpha_\ep^{1-\gamma} + |u-w|^{1-\gamma}}^\frac{\theta\xi}{\theta-2+p}\chi_{\mathfrak{B}_{\ep}}}{dx dt}}^\frac{\theta-2+p}{\theta}.
\end{aligned}
\end{equation*}
In addition, from the following relation
$$
\theta < p - \frac{n}{n+1} \iff \frac{(1-\gamma)\theta}{\theta-2+p} < \frac{(n+1)p-n\beta}{np}=\frac{\theta (n+1)}{n},
$$
we can take $\xi>1$ so that $\frac{(1-\gamma)\theta\xi}{\theta-2+p} < \frac{\theta (n+1)}{n}$, and then we discover
\begin{equation*}
\begin{aligned}
&\gh{\mint{\ik{8r}}{\gh{\alpha_\ep^{1-\gamma} + |u-w|^{1-\gamma}}^\frac{\theta\xi}{\theta-2+p}\chi_{\mathfrak{B}_{\ep}}}{dx dt}}^\frac{\theta-2+p}{\theta} \\
&\qquad \leq c \alpha_\ep^{(1-\gamma)\xi}+ c \gh{\mint{\ik{8r}}{|u-w|^\frac{\theta (n+1)}{n}\chi_{\mathfrak{B}_{\ep}}}{dxdt}}^\frac{(1-\gamma)\xi n}{\theta(n+1)} \leq c \alpha_\ep^{(1-\gamma)\xi}
\end{aligned}
\end{equation*}
by using \eqref{st2-2}.
Thus we have
\begin{equation}\label{st2-5}
\begin{aligned}
J_2 \leq c \alpha_\ep^\frac{p-2\beta}{2p} \bgh{\frac{|\mu|(\ik{8r})}{\norm{\ik{8r}}}}^\frac{1}{2}\gh{\mint{\ik{8r}}{|Du|^\theta}{dx dt}}^\frac{2-p}{2\theta}.
\end{aligned}
\end{equation}

If $\lim_{\ep\to 0} M_\ep=0$, then $Du\equiv Dw$ a.e. in $\ik{8r}$, and so the proof is done. Thus, we may assume
$\inf _{\ep>0} M_\ep>0$, and then there exists a constant $\ep_0>0$ such that
$0<\ep < \bgh{|\mu|(\ik{8r})M_\ep^n}^\frac{1}{\theta (n+1)}$
whenever $0<\ep<\ep_0$. Consequently, we see from \eqref{def of alpha} that
\begin{equation}\label{st2-4.1}
\alpha_\ep < 2\bgh{|\mu|(\ik{8r})M_\ep^n}^\frac{1}{\theta (n+1)} \quad \text{for all} \ \  0<\ep<\ep_0.
\end{equation}
Inserting \eqref{st2-4.5} and \eqref{st2-5} into \eqref{st2-3}, we employ \eqref{st2-4.1} and Young's inequality to discover
\begin{equation}\label{st2-6}
\begin{aligned}
M_\ep &\leq c\bgh{\frac{|\mu|(\ik{8r})}{\norm{\ik{8r}}}}^\frac{\theta(n+1)}{(n+1)p-n}\bgh{|\mu|(\ik{8r})}^\frac{1-\beta}{(n+1)p-n}\\
&\quad + c\bgh{\frac{|\mu|(\ik{8r})}{\norm{\ik{8r}}}}^\frac{\theta(n+1)}{n+2}\bgh{|\mu|(\ik{8r})}^\frac{p-2\beta}{(n+2)p}\gh{\mint{\ik{8r}}{|Du|^\theta}{dx dt}}^\frac{(2-p)(n+1)}{n+2}.
\end{aligned}
\end{equation}

Finally, we combine \eqref{st2-2.5}, \eqref{st2-4.1} and \eqref{st2-6} to obtain
\begin{equation*}
\begin{aligned}
\gh{\mint{\ik{8r}}{|Du-Dw|^{\theta}\chi_{\mathfrak{B}_{\ep}}}{dxdt}}^{\frac{1}{\theta}} &\le c \bgh{ \frac{|\mu|(\ik{8r})}{|\ik{8r}|^{\frac{n+1}{n+2}}} }^{\frac{n+2}{(n+1)p-n}}\\
&\quad + c \bgh{ \frac{|\mu|(\ik{8r})}{|\ik{8r}|^{\frac{n+1}{n+2}}} } \gh{\mint{\ik{8r}}{|Du|^{\theta}}{dxdt}}^{\frac{(2-p)(n+1)}{\theta(n+2)}}
\end{aligned}
\end{equation*}
whenever  $0<\ep<\ep_0$. Letting $\ep\to 0$, we obtain the desired estimate \eqref{compa1-r}.
\end{proof}

\begin{remark}\label{compa rmk1}
The approach used in the proof above is motivated by \cite{NP19,KM14b,KM13b}. This approach is applicable to comparison estimates for elliptic problems (with $\frac{3n-2}{2n-1} < p \le 2-\frac{1}{n}$), which indeed gives a different way from that of \cite[Lemma 2.2]{NP19}. On the other hand, for comparison estimates like \eqref{compa1-r} with $p>2-\frac{1}{n+1}$, we refer to \cite[Lemma 4.1]{KM14b} and \cite[Lemma 4.3]{KM13b}.
\end{remark}

The next lemma gives a boundary self-improving result for $Dw$ (see \cite{KL00} for an interior version).

\begin{lemma}\label{high int}
Let $\frac{2n}{n+2}<p\le2$ and let $\frac{(2-p)n}{2}<\theta\le p$. If $w$ is the weak solution of \eqref{pem2} satisfying \eqref{reifen} and
\begin{equation}\label{hi-c}
\mint{\ik{8r}}{|Dw|^{\theta}}{dxdt} \le c_w \lambda^{\theta}
\end{equation}
for some constant $c_w\ge1$, then there exist two constants $\sigma=\sigma(n,\La_0,\La_1,p,\theta)>0$ and $c=c(n,\La_0,\La_1,p,\theta,c_w)\ge1$ such that
\begin{equation}\label{hi-r}
\mint{\ik{4r}}{|Dw|^{p(1+\sigma)}}{dx dt} \le c\lambda^{p(1+\sigma)}.
\end{equation}
\end{lemma}

\begin{proof}
From \cite[Theorem 2.2]{BP10} and \cite[Remark 6.12]{Giu03} (see also \cite[Lemma 4.2]{BPS21}), we infer
\begin{equation}\label{hi-5}
\mint{\ik{4r}}{|Dw|^{p(1+\sigma)}}{dxdt} \le c \la^{p(1+\sigma)}\bgh{\gh{\la^{-ps} \mint{\ik{8r}}{|Dw|^{ps}}{dxdt}}^{\frac{1+\mathfrak{d}\sigma}{1-\mathfrak{d}+\mathfrak{d}s}} + 1}
\end{equation}
for every $s \in \left(\frac{(2-p)n}{2p},1\right]$, where $c=c(n,\La_0,\La_1,p,s)\ge1$ and $\mathfrak{d}:=\frac{2p}{(n+2)p-2n}$.
By taking $s=\frac{\theta}{p}$ and using \eqref{hi-c}, we obtain the desired estimate \eqref{hi-r}.
\end{proof}

\begin{remark}
\begin{enumerate}[(i)]
\item It is worth noting that the scaling deficit $\mathfrak{d}$ and interpolation inequality give the lower bound of $s$ in \eqref{hi-5}, which determines the range of $\theta$. For elliptic equations with $p$-growth $(p>1)$, on the other hand, the estimate like \eqref{hi-5} holds for every $s\in (0,1]$ (see \cite[Remark 6.12]{Giu03}).

\item Lemma \ref{high int} also holds for the case $p\ge2$ under an appropriate range of $\theta$ (see \cite{KL00,BP10,BPS21}).
\end{enumerate}
\end{remark}

Let us now consider the unique weak solution $v$ to the coefficient frozen problem
\begin{equation}
\label{pem3}\left\{
\begin{alignedat}{3}
v_t -\ddiv \bar{\ma}_{B_{4r}^+}(Dv,t) &= 0 &&\quad \text{in} \ K_{4r}^{\lambda}, \\
v &= w &&\quad \text{on} \ \partial_p K_{4r}^{\lambda},
\end{alignedat}\right.
\end{equation}
where a freezing operator $\bar{\ma}_{B_{4r}^+} = \bar{\ma}_{B_{4r}^+}(\xi,t) : \bb^n \times I_{4r}^\la \to \bb^n$ is given by
\begin{equation*}
\bar{\mathbf{a}}_{B_{4r}^+}(\xi,t) := \mint{B_{4r}^+}{\mathbf{a}(\xi,x,t)}{dx}.
\end{equation*}

We derive the following comparison result between \eqref{pem2} and \eqref{pem3}:
\begin{lemma}\label{fcompa1}
Let $p>\frac{2n}{n+2}$, let $w$ be the weak solution of \eqref{pem2} satisfying \eqref{reifen} and \eqref{hi-c}, and let $v$ as in \eqref{pem3}. Then there is a constant
$c=c(n,\La_0,\La_1,p)\ge1$ such that
\begin{equation*}\label{fcompa1_r}
\mint{\ik{4r}}{|Dw-Dv|^{p}}{dx dt} \le c \delta^{\sigma_1} \la^p,
\end{equation*}
where $\sigma_1=\sigma_1(n,\La_0,\La_1,p)>0$.
\end{lemma}

\begin{proof}
The proof follows from Lemma \ref{high int}, \eqref{2-small BMO2}, \eqref{measure density} and \cite[Lemma 3.10]{BOR13}.
\end{proof}

For interior regularity results (see \cite{DF85, DF85b, DiB93}),  we see $Dv \in L_{loc}^{\infty}(\iq{4r})$ in the interior region $\gh{\iq{4r} \subset \OO_T}$. On the other hand, for the boundary case $\gh{\iq{4r}  \not\subset \OO_T}$, the $L^\infty$-norm of $Dv$ could not be bounded when $\partial\OO$ is very irregular. Thus, we need to consider a weak solution $\bar{v}$ to the following problem:
\begin{equation}
\label{pem4}\left\{
\begin{alignedat}{3}
\bar{v}_t -\ddiv \bar{\mathbf{a}}_{B_{4r}^{+}}(D\bar{v},t) &= 0 &&\quad \text{in} \ Q_{2r}^{\lambda,+}, \\
\bar{v} &= 0 &&\quad \text{on} \ T_{2r}^\la.
\end{alignedat}\right.
\end{equation}

We recall the boundedness of $D\bar{v}$ near the flat boundary, as follows:
\begin{lemma}[See {\cite[Theorem 1.6]{Lie93}}]
\label{bdy Lip reg}
Let $p>\frac{2n}{n+2}$. For any weak solution $\bar{v}$ of \eqref{pem4}, we have
\begin{equation*}
\label{fcompa2_r2}
\Norm{D\bar{v}}^p_{L^{\infty}(Q_{r}^{\lambda,+})} \le c \mint{Q_{2r}^{\lambda,+}}{|D\bar{v}|^p}{dx dt} +c\lambda^p
\end{equation*}
for some constant $c=c(n,\La_0,\La_1,p) \ge 1$.
\end{lemma}

If the boundary of $\OO$ is sufficiently flat in the sense of $(\delta,R)$-Reifenberg domain, for some appropriate weak solution $\bar{v}$ of \eqref{pem4} we have a comparison estimate between \eqref{pem3} and \eqref{pem4} as follows:
\begin{lemma}\label{fcompa2}
Let $p>\frac{2n}{n+2}$. For any $\varepsilon \in (0,1)$, there exists a small constant
$\delta=\delta(n,\La_0,\La_1,p,\varepsilon)>0$ such that the following holds:
if $v$ is the weak solution of \eqref{pem3} satisfying \eqref{reifen} and
\begin{equation*}\label{fcompa2-c}
\mint{\ik{4r}}{|Dv|^p}{dxdt} \le c_v \lambda^p
\end{equation*}
for some given constant $c_v \ge1$, then there is a weak solution $\bar{v}$ of \eqref{pem4} such that
\begin{equation}\label{fcompa2_r1}
\mint{\ik{2r}}{|Dv-D\bar{v}|^p}{dx dt} \le \varepsilon^p \la^p \quad \text{and} \quad \mint{\ik{2r}}{|D\bar{v}|^p}{dx dt} \le c \la^p
\end{equation}
for some constant $c=c(n,\La_0,\La_1,p,c_v)\ge 1$, where $\bar{v}$ is extended by zero from $\iqq{2r}$ to $\ik{2r}$.
\end{lemma}

\begin{proof}
The first estimate in \eqref{fcompa2_r1} comes from the compactness argument as in \cite[Lemma 3.8]{BOR13}. It follows from this first estimate and \eqref{measure density} that the second estimate in  \eqref{fcompa2_r1} holds.
\end{proof}

Finally, combining all the previous results, we directly obtain the desired local comparison estimate below $L^1$ spaces near the boundary of $\OO$.
\begin{proposition}\label{comparison}
Let $\frac{2n}{n+1}<p\le 2-\frac{1}{n+1}$ and let $\max\mgh{\frac{n+2}{2(n+1)},\frac{(2-p)n}{2}} < \theta < p-\frac{n}{n+1} \le 1$. For any $\varepsilon  \in (0,1)$, there is a small constant $\delta=\delta(n,\La_0,\La_1,p,\theta,\varepsilon) >0$ such that the following holds: if $u$ is a weak solution of \eqref{tweaksol} satisfying \eqref{reifen},
\begin{equation}\label{comparison-c}
\mint{\ik{8r}}{|Du|^\theta}{dxdt} \le \la^\theta \quad \text{and} \quad \frac{|\mu|(\ik{8r})}{r^{n+1}} \le \delta\la,
\end{equation}
then there exists a weak solution $\bar{v}$ of \eqref{pem4} such that
\begin{equation*}\label{comparison-r}
\mint{K^\lambda_{r}}{|Du-D\bar{v}|^\theta}{dxdt} \leq \varepsilon \lambda^\theta \quad\text{and}\quad \|D\bar{v}\|_{L^\infty(K^{\lambda}_r)} \leq c \lambda
\end{equation*}
for some constant $c=c(n,\La_0,\La_1,p,\theta)\ge 1$, where $\bar{v}$ is extended by zero from $\iqq{2r}$ to $\ik{2r}$.
\end{proposition}

%\begin{proof}
%It follows from \eqref{comparison-c} and Lemma \ref{compa1} that
%\begin{equation}\label{comparison-1}
%\mint{\ik{8r}}{|Du-Dw|^\theta}{dxdt} \le c\delta^{\theta}\la^\theta \ \text{and} \ \mint{\ik{8r}}{|Dw|^\theta}{dxdt} \le c_w\la^\theta
%\end{equation}
%for some constant $c_w=c_w(n,\La_0,\theta)\ge1$.
%Combining Lemmas \ref{high int} and \ref{fcompa1} yields
%\begin{equation}\label{comparison-2}
%\begin{aligned}
%\mint{\ik{4r}}{|Dv|^p}{dx dt} &\le 2^{p-1} \mgh{\mint{\ik{4r}}{|Dw-Dv|^{p}}{dx dt} + \mint{\ik{4r}}{|Dw|^{p}}{dx dt}}\\
%&\le c \delta^{\sigma_1} \la^p +  c\la^p =: c_v\la^p.
%\end{aligned}
%\end{equation}
%Applying Lemma \ref{fcompa2} with $\ep$ replaced by $\tilde{\ep}$, we find a weak solution $\bar{v}$ of \eqref{pem4} such that
%\begin{equation}\label{comparison-3}
%\gh{\mint{\ik{2r}}{|Dv-D\bar{v}|^\theta}{dx dt}}^{\frac{1}{\theta}} \le \gh{\mint{\ik{2r}}{|Dv-D\bar{v}|^p}{dx dt}}^{\frac{1}{p}} \le \tilde{\ep} \la \le \frac{\ep}{3} \la,
%\end{equation}
%by selecting $\tilde{\ep}$ with $0<\tilde{\ep} \le \frac{\ep}{3}$.
%Then we employ \eqref{measure density} and \eqref{comparison-1}--\eqref{comparison-3} to estimate
%\begin{equation*}
%\begin{aligned}
%\mint{K^\lambda_{r}}{|Du-D\bar{v}|^\theta}{dxdt} &\le c \mint{K^\lambda_{r}}{|Du-Dw|^\theta+|Dw-Dv|^\theta+|Dv-D\bar{v}|^\theta}{dxdt}\\
%&\le c\delta^{\theta}\la^\theta + c \delta^{\frac{\sigma_1 \theta}{p}} \la^\theta + \gh{\frac{\ep}{3} \la}^\theta\\
%&\le \ep\la^\theta,
%\end{aligned}
%\end{equation*}
%by choosing $\delta$ sufficiently small.
%
%On the other hand, the second estimate in \eqref{comparison-r} follows directly from Lemmas \ref{bdy Lip reg} and \ref{fcompa2}.
%\end{proof}

\begin{remark}\label{comparmk}
\begin{enumerate}[(i)]
\item In view of Lemmas \ref{compa1} and \ref{high int}, the valid range of $p$ in Proposition \ref{comparison} is $\frac{2n}{n+1}<p\le 2-\frac{1}{n+1}$, not $\max\mgh{\frac{3n+2}{2n+2},\frac{2n}{n+2}}<p\le 2-\frac{1}{n+1}$, since the constant $\theta$ exists only when $p-\frac{n}{n+1}>\max\mgh{\frac{n+2}{2(n+1)},\frac{(2-p)n}{2}}$. Note that $\frac{2n}{n+1} \ge \max\mgh{\frac{3n+2}{2n+2},\frac{2n}{n+2}}$, where the equality holds if and only if $n=2$.

\item We note that the value $\frac{|\mu|(\ik{8r})}{r^{n+1}}$ in \eqref{comparison-c} is related to the intrinsic fractional maximal function $\M_1^\la(\mu)$ defined in \eqref{intrinsic fractional mf}, see Section \ref{la covering arguments} below.
\end{enumerate}
\end{remark}

Similarly, we have an interior comparison estimate below $L^1$ spaces.
\begin{proposition}\label{comparison2}
Let $\frac{2n}{n+1}<p\le 2-\frac{1}{n+1}$ and let $\max\mgh{\frac{n+2}{2(n+1)},\frac{(2-p)n}{2}} < \theta < p-\frac{n}{n+1} \le 1$. For any $\varepsilon  \in (0,1)$, there is a small constant $\delta=\delta(n,\La_0,\La_1,p,\theta,\varepsilon) >0$ such that the following holds: if $u$ is a weak solution of \eqref{tweaksol} satisfying
\begin{equation*}
\mint{\iq{8r}}{|Du|^\theta}{dxdt} \le \la^\theta \quad \text{and} \quad \frac{|\mu|(\iq{8r})}{r^{n+1}} \le \delta\la,
\end{equation*}
then there exists a weak solution $v$ of \eqref{pem3} such that
\begin{equation*}
\mint{Q^\lambda_{r}}{|Du-Dv|^\theta}{dxdt} \leq \varepsilon \lambda^\theta \quad\text{and}\quad \|Dv\|_{L^\infty(Q^{\lambda}_r)} \leq c \lambda
\end{equation*}
for some constant $c=c(n,\La_0,\La_1,p,\theta)\ge 1$.
\end{proposition}

%%%%%%%%%%%%%%%%%%%%%%%%%%%%%%%%%%%%%%%%%%%%%%%%%%%%%%%%
\section{$\lambda$-covering arguments}\label{la covering arguments}

We now consider a renormalized solution $u$ of the problem \eqref{pem1}. We denote by $u_k:=T_k(u)$ ($k \in \mathbb{N}$) the truncation of $u$ and $\mu_k \in L^{p'}(0,T; W^{-1,p'}(\OO))$ the corresponding measure given in \eqref{tweaksol}.
%Regarding $\mu_k$ in \eqref{tweaksol} as an approximation of $\mu$ in \eqref{pem1} via mollification backward in time and then using a suitable cutoff function in time, $\mu_k\in L^\infty(\omt)$ satisfies the properties \eqref{sola2}, \eqref{sola3} and \eqref{appro mu}; then one can apply all the results obtained in Section \ref{intrinsic comparison} to $u=u_k$ and $\mu=\mu_k$.
We also denote by $w_k$, $v_k$ and $\bar{v}_k$ the corresponding weak solutions of \eqref{pem2}, \eqref{pem3} and \eqref{pem4}, respectively. The goal of this section is to derive an appropriate decay estimate for the upper level set of the $\lambda$-maximal function of $|Du|^\theta$ (see Proposition \ref{covarg} later).

We note that $\mu_k=\mu_a^+-\mu_a^-+\nu_k^+-\nu_k^-$ for $k\in \mathbb{N}$. Since $\mu_a^\pm+\nu_k^\pm \to \mu_a^\pm+\mu_s^\pm$ tightly as $k \to \infty$, we have %the sequence of nonnegative measures $\mu_a^\pm+\nu_k^\pm$ converges to $\mu_a^\pm+\mu_s^\pm$ weakly in $\mathfrak{M}_b(\OO_T)$, whence we have
\begin{equation}\label{est-mu}
\limsup_{k \to \infty} |\mu_k|(K\cap \OO_T) \leq |\mu|(K\cap \OO_T)
\end{equation}
for every compact set $K\subset \mathbb{R}^{n+1}$. Having this relation in mind, we start with a standard energy type estimate for \eqref{pem1} as follows:
\begin{lemma}\label{standard estimate}
Let $\frac{3n+2}{2n+2}<p \le 2-\frac{1}{n+1}$. If $u$ is a renormalized solution of \eqref{pem1}, then there exists a constant $c = c (n,\La_0,p,\theta) \ge 1$ such that
\begin{equation}\label{ste1}
\gh{\mint{\OO_T}{|Du|^\theta}{dx dt}}^{\frac{1}{\theta}} \le c \bgh{\frac{|\mu|(\OO_T)}{|\OO_T|^{\frac{n+1}{n+2}}}}^{\frac{n+2}{(n+1)p-n}}
\end{equation}
for any constant $\theta$ such that $0 < \theta < p - \frac{n}{n+1}$.
\end{lemma}

\begin{proof}
Since $\mathbf{a}(0,x,t)=0$, the zero function $\bar{w}$ solves the Cauchy-Dirichlet problem
\begin{equation*}
\left\{
\begin{alignedat}{3}
\bar{w}_t -\ddiv \mathbf{a}(D\bar{w},x,t) &= 0 &&\quad \text{in} \ \OO_T, \\
\bar{w} &= 0 &&\quad\text{on} \ \partial_p \OO_T.
\end{alignedat}\right.
\end{equation*}
Replacing $u$, $w$, $\mu$ and $\ik{8r}$ by $u_k$, $\bar{w}(\equiv0)$, $\mu_k$ and $\OO_T$, respectively in the proof of Lemma \ref{compa1}, we deduce
\begin{equation}\label{ste2}
\gh{\mint{\OO_T}{|Du_k|^\theta}{dx dt}}^{\frac{1}{\theta}} \le c \bgh{\frac{|\mu_k|(\OO_T)}{|\OO_T|^{\frac{n+1}{n+2}}}}^{\frac{n+2}{(n+1)p-n}}
\end{equation}
whenever $\frac{n+2}{2(n+1)} < \theta < p - \frac{n}{n+1}$. Also, it follows from  \eqref{est-mu} that
$$
\limsup_{k \to \infty} |\mu_k|(\OO_T) \leq |\mu|(\OO_T).
$$
Taking the limit supremum of both sides of \eqref{ste2}, we see that \eqref{ste1} holds for all $\theta \in \gh{\frac{n+2}{2(n+1)} , p - \frac{n}{n+1}}$. On the other hand, if $ \theta \in \left(0, \frac{n+2}{2(n+1)}\right]$, we employ H\"older's inequality to obtain the desired estimate.
\end{proof}

We next introduce a modified version of Vitali's covering lemma for intrinsic parabolic cylinders, which is an important tool for obtaining Proposition \ref{covarg} later.
\begin{lemma}[See {\cite[Lemma 2.14]{BPS21}}]
\label{VitCov}
Let $0<\varepsilon<1$, let $\lambda>0$, and let $\omt := \OO \times (-\infty,T)$, where $\Omega$ is $(\delta,R)$-Reifenberg flat. Let $\C\subset \D \subset \omt$ be two bounded measurable subsets such that
\begin{enumerate}[(i)]
\item\label{vic1} $|\C|<\varepsilon \norm{Q^\lambda_{R/10}}$, and
\item\label{vic2} for any $(y,s) \in \omt$ and any $r \in \left(0,\frac{R}{10}\right]$ with $|\C\cap Q^\lambda_r(y,s)|\geq \varepsilon |Q^\lambda_r|$, \\
$Q^\lambda_r(y,s)\cap \omt \subset \D$.
\end{enumerate}
Then we have
\begin{equation}\label{vitcov-r}
|\C|\leq \left(\frac{10}{1-\delta}\right)^{n+2}\varepsilon|\D| \leq \left(\frac{80}{7}\right)^{n+2}\varepsilon|\D|.
\end{equation}
\end{lemma}

\begin{remark}\label{vicrmk1}
\begin{enumerate}[(i)]
\item It is worth noting that  the covering lemma above is obtained under $\C\subset \D \subset \omt$, not $\C\subset \D \subset \OO_T := \OO \times (0,T)$; thereby the relation \eqref{vitcov-r} is independent of $\lambda$ (see \cite[Remark 2.15]{BPS21} for details). For this reason, we considered the comparison estimates in Section \ref{intrinsic comparison} on the localized region of $\omt$, not $\OO_T$.

\item For Vitali's covering lemmas with respect to standard balls or cubes, we refer to for instance \cite{CP98,Wan03,BW04,Min10}.
\end{enumerate}
\end{remark}

%\begin{remark}\label{vicrmk2}
%{\color{dgreen} \bf
%When considering weak solution of the quasilinear parabolic problems without measure data, We can extend $\OO_T$ to $\OO \times \bb$ by $t\ge T$, extend $f=0$ and $t<0, u \equiv 0$, see \cite{BOR13, Byun's papers}!!!!
%and then we only need to consider the lateral boundary. However, in case of measure data problems, since there is no uniqueness of renormalized solution in general, we cannot extend a renormalized solution $u$ of \eqref{pem1} for $t\ge T$, see also \eqref{ste1}!!!!!
%}
%\end{remark}

%\begin{remark}\label{vicrmk3}
%A shape of the parabolic intrinsic cylinders $\iq{r}(x,t)$ is essential to prove Lemma \ref{VitCov}; if we take the intrinsic cylinders having the top center such as $B_r(x) \times (t- \lambda^{2-p} r^2, t)$ instead of $\iq{r}(x,t)$, then the standard Vitali covering lemma no longer holds, see \cite[Theorem C.1]{Bog07}.
%\end{remark}

Suppose that $(\mathbf{a},\OO)$ is $(\delta,R)$-vanishing and that $p$ satisfies \eqref{p-range}, unless otherwise stated.
We write
\begin{equation*}\label{dc}
\be := \frac{n+2}{(n+1)p-n} \quad \text{and} \quad d := \frac{2}{(n+1)p-2n}.
\end{equation*}
For any fixed $\ep \in (0,1)$, we set
\begin{equation}
\label{55}
\lambda_0 := \bgh{\frac{|\mu|(\OO_T)}{|\OO_T|^{\frac{n+1}{n+2}}}}^{\be\theta} \frac{|\OO_T|}{\varepsilon \left|Q_{R/10}\right|} + \bgh{\frac{|\mu|(\OO_T)}{\delta T^{\frac{n+1}{2}}}}^{d} +1,
\end{equation}
where $\theta$ is a constant such that
\begin{equation}\label{ka-range}
\max\mgh{\frac{n+2}{2(n+1)},\frac{(2-p)n}{2},p-1} < \theta < p-\frac{n}{n+1} \le 1.
\end{equation}
We remark that the constant $\beta$ arises in Lemma \ref{standard estimate}, $d$ in Lemma \ref{difference of ratio}, and $\theta$ in Proposition \ref{comparison} and Lemma \ref{DecayLem}.    We may assume, upon letting $u \equiv 0$ for $t < 0$, that a renormalized solution $u$ is defined in $\omt := \OO \times (-\infty,T)$. For any fixed $N>1$ and $\la \ge \la_0 \ge1$, we write
\begin{equation*}
\C := \left\{ (x,t)\in\omt: \m_{\omt}^\lambda|Du|^\theta (x,t)> \gh{N\lambda}^\theta \right\}
\end{equation*}
and
\begin{equation*}
\D := \left\{(x,t)\in\omt:\m_{\omt}^\lambda|Du|^\theta(x,t) > \lambda^\theta \right\}\cup\left\{(x,t)\in\omt:\m^\lambda_1(\mu)(x,t)>\delta \lambda \right\},
\end{equation*}
where $\m_{\omt}^\lambda$ is given by \eqref{intrinsic mf}, and the operator $\m^\lambda_1$ is the {\em intrinsic fractional maximal function of order $1$ for $\mu$} defined by
\begin{equation}\label{intrinsic fractional mf}
\m_1^\lambda(\mu) (x,t) := \sup_{r>0} \frac{ |\mu|(\iq{r}(x,t))}{r^{n+1}} \quad \text{for} \ \ (x,t) \in \bb^{n+1}.
\end{equation}
From Lemma \ref{standard estimate}, the $\lambda$-maximal function $\m_{\omt}^\lambda|Du|^\theta(x,t)$ is well defined for all $(x,t)\in \omt$.
We note that since the support of $|Du|$ is bounded,  both $N\la$-upper level set and $\la$-upper level set of $\m_{\omt}^\lambda|Du|^\theta$ should be bounded. Combining this fact and Lemma \ref{difference of ratio}, we infer that both the upper level sets $\C$ and $\D$ are \textit{bounded} measurable subsets of $\omt$.

Under these settings, we first prove two assumptions of Lemma \ref{VitCov}.
\begin{lemma}\label{DecayLem}
There exists  a constant $N_1=N_1(n,\La_0,p,\theta)>1$ such that for any fixed $N\ge N_1$ and $\la \ge \la_0$, we have
\begin{align*}
\left|\C\right| < \varepsilon\left|Q_{R/10}^\lambda\right|.
\end{align*}
\end{lemma}

\begin{proof}
Since $\theta>p-1$ and $\la \ge \la_0\ge1$, we note that $\la^{-\theta} \le \la^{1-p} \le \la^{2-p} \la_0^{-1}$. Then we have from  \eqref{lambda 1-1}, Lemma \ref{standard estimate} and \eqref{55} that
\begin{align*}
|\C| &\le \frac{c}{\gh{\lambda N}^\theta} \int_{\OO_T} |Du|^{\theta}\ dxdt \le \frac{c |\OO_T|}{\gh{\lambda N}^\theta}\bgh{\frac{|\mu|(\OO_T)}{|\OO_T|^{\frac{n+1}{n+2}}}}^{\beta\theta}\\
&< \frac{c\lambda^{2-p}\varepsilon \left|Q_{R/10}\right|}{{N_1}^{\theta}} \le \varepsilon\left|Q_{R/10}^\lambda\right|,
\end{align*}
by selecting $N_1$ large enough.
\end{proof}

\begin{lemma}
\label{2_lemma}
For any $\varepsilon\in(0,1)$, there exist $N_2=N_2(n,\La_0,\La_1, p,\theta)>1$ and $\delta=\delta(n,\La_0,\La_1,p,\theta,\varepsilon) \in \left(0, \frac{1}{8}\right)$ such that the following holds: for any fixed  $\lambda \ge \la_0$, $N \ge N_2$, $r\in \left(0,\frac{R}{10}\right]$ and $(y,s)\in \omt$ with
\begin{equation}
\label{51}
\left|\C \cap Q_r^\lambda(y,s)\right| \ge \varepsilon\left|Q_r^\lambda\right|,
\end{equation}
we have
\begin{equation*}
K_r^\lambda(y,s) \subset \D.
\end{equation*}
\end{lemma}

\begin{proof}
We argue by a contradiction. Assume $K_r^\lambda(y,s)  \not\subset \D$. Then there is a point $(\tilde{x},\tilde{t}) \in K_r^\lambda(y,s)$ such that for all $\rho>0$,
\begin{equation}
\label{52}
\frac{1}{\left|Q^\lambda_{\rho}\right|}\int_{K_{\rho}^\lambda(\tilde{x},\tilde{t})} |Du|^{\theta}\ dxdt \leq \lambda^{\theta} \quad \text{and} \quad \frac{|\mu|(K^\lambda_{\rho}(\tilde{x},\tilde{t}))}{\rho^{n+1}} \leq \delta\lambda.
\end{equation}
We split the proof into an interior case ($Q_{32r}^\lambda(\tilde{x},\tilde{t}) \subset \omt$) and a boundary case ($Q_{32r}^\lambda(\tilde{x},\tilde{t}) \not\subset \omt$). In this proof, we only consider the boundary case. We can similarly prove the interior case by using Proposition~\ref{comparison2} instead of Proposition~\ref{comparison}.
Note that $\pc{K_{\rho}^\lambda(y,s)} := K_{\rho}^\lambda(y,s) \cup \partial_p K_{\rho}^\lambda(y,s) \subset K_{\rho+r}^\lambda(\tilde{x},\tilde{t})$ for any $\rho > r$.  Then \eqref{measure density} and \eqref{52} yield
\begin{align*}
\mint{K^\lambda_\rho(y,s)}{|Du|^{\theta}}{dxdt} &\leq  \frac{c}{|Q^\lambda_\rho|}\int_{K^\lambda_\rho(y,s)} |Du|^{\theta}\ dxdt \\
&\leq \frac{c}{|Q^\lambda_{\rho+r}|}\int_{K^\lambda_{\rho+r}(\tilde{x},\tilde{t})} |Du|^{\theta}\ dxdt \leq c_2 \lambda^{\theta}
\end{align*}
whenever $\rho\geq 32r$. Since $u_k$ is the truncation of $u$, we have
\begin{equation}\label{co2-1}
\begin{aligned}
\mint{K^\lambda_\rho(y,s)}{|Du_k|^{\theta}}{dxdt} \le \mint{K^\lambda_\rho(y,s)}{|Du|^{\theta}}{dxdt} \le  c_2 \la^{\theta}
\end{aligned}
\end{equation}
for any $k\in \mathbb{N}$. Combining \eqref{est-mu} and \eqref{52}, we also deduce
\begin{equation}\label{co2-2}
\begin{aligned}
\frac{|\mu_k|(K^\lambda_\rho(y,s))}{\rho^{n+1}} \le \frac{2 |\mu|(\pc{K^\lambda_\rho(y,s)})}{\rho^{n+1}}\le \frac{c_2|\mu|(K^\lambda_{\rho+r}(\tilde{x},\tilde{t}))}{(\rho+r)^{n+1}}  \le c_2\delta\lambda
\end{aligned}
\end{equation}
whenever $\rho \geq 32r$ and $k$ is sufficiently large. We apply Proposition \ref{comparison} with $u$, $\mu$, $(x_0,t_0)$, $\la$, $r$ and $\ep$ replaced by $u_k$, $\mu_k$, $(y,s)$, $c_2 \la$, $4r$ and $\eta$, respectively. Then for any $\eta  \in (0,1)$, there is a small constant $\delta=\delta(n,\La_0,\La_1,p,\theta,\eta) >0$ such that the following holds: if $u_k$ is a weak solution of \eqref{tweaksol} satisfying \eqref{co2-1} and \eqref{co2-2}, then there exists a corresponding weak solution $\bar{v}_k$ of \eqref{pem4} such that
\begin{equation}\label{co2-3-1}
\mint{K^\lambda_{4r}(y,s)}{|Du_k-D\bar{v}_k|^{\theta}}{dxdt} \leq   \eta (c_2 \lambda)^{\theta}
\end{equation}
and
\begin{equation}\label{co2-3-2}
\|D\bar{v}_k\|_{L^\infty\gh{K^{\lambda}_{4r}(y,s)}} \leq c c_2 \lambda =: c_3 \la
\end{equation}
for some constant $c_3 = c_3(n,\La_0,\La_1, p,\theta)\ge1$.
Using Lemma \ref{standard estimate} and the absolute continuity of the Lebesgue integral, we find
\begin{equation*}
\mint{K^\lambda_{4r}(y,s)}{|Du-Du_k|^{\theta}}{dxdt}  = \mint{K^\lambda_{4r}(y,s)}{\chi_{\mgh{|Du|>k}}|Du|^{\theta}}{dxdt} \leq \eta (c_2 \lambda)^{\theta}
\end{equation*}
for sufficiently large $k$. Combining this inequality and \eqref{co2-3-1}, we discover
\begin{equation}\label{co2-4}
\mint{K^\lambda_{4r}(y,s)}{|Du-D\bar{v}_k|^{\theta}}{dxdt} \le c_4 \eta \la^{\theta}.
\end{equation}

We next show that
\begin{equation}\label{53}
\begin{aligned}
&\left\{ (x,t)\in K_r^\lambda(y,s):\m_{\omt}^\lambda|Du|^{\theta} > \gh{N\lambda}^{\theta} \right\}\\
&\qquad \subset \left\{(x,t)\in K_r^\lambda(y,s):\m_{K_{4r}^\lambda(y,s)}^\lambda  |Du-D\bar{v}_k|^{\theta} >\lambda^{\theta} \right\}
\end{aligned}
\end{equation}
provided $N \ge N_2 :=  \max\left\{2^{n+2}, 1+{c_3}^{\theta}\right\}$. To do so, let
$$(\tilde{y},\tilde{s}) \in \left\{ (x,t)\in K_r^\lambda(y,s) : \m_{K_{4r}^\lambda(y,s)}^\lambda |Du-D\bar{v}_k|^{\theta}\leq \lambda^{\theta} \right\}.$$
Then for any $\tilde{r}>0$,
\begin{equation}
\label{54}
\frac{1}{\left|Q_{\tilde{r}}^\lambda\right|}\int_{K^\lambda_{\tilde{r}}(\tilde{y},\tilde{s})\cap K_{4r}^\lambda(y,s)} |Du-D\bar{v}_k|^{\theta}\; dxdt \leq \lambda^{\theta}.
\end{equation}
If $\tilde{r}\in(0,2r]$, then $K_{\tilde{r}}^\lambda(\tilde{y},\tilde{s})\subset K^\lambda_{3r}(y,s)$. It follows from \eqref{54} and \eqref{co2-3-2} that
\begin{align*}
\frac{1}{\left|Q_{\tilde{r}}^\lambda\right|}\int_{K^\lambda_{\tilde{r}}(\tilde{y},\tilde{s})}
|Du|^{\theta}\ dxdt &\leq \frac{1}{\left|Q_{\tilde{r}}^\lambda\right|}\int_{K^\lambda_{\tilde{r}}(\tilde{y},\tilde{s})} \gh{|Du-D\bar{v}_k|^{\theta}+|D\bar{v}_k|^{\theta}}\ dxdt\\
&\le \gh{1+  {c_3}^{\theta}} \lambda^{\theta}.
\end{align*}
If $\tilde{r}>2r$, then $K^\lambda_{\tilde{r}}(\tilde{y},\tilde{s})\subset K^\lambda_{\tilde{r}+r}(y,s)\subset K^\lambda_{2\tilde{r}}(\tilde{x},\tilde{t})$. We use the first inequality of \eqref{52} to obtain
\begin{align*}
\frac{1}{\left|Q_{\tilde{r}}^\lambda\right|}\int_{K^\lambda_{\tilde{r}}(\tilde{y},\tilde{s})}
|Du|^{\theta}\ dxdt \leq \frac{1}{\left|Q_{\tilde{r}}^\lambda\right|} \int_{K^\lambda_{2\tilde{r}}(\tilde{x},\tilde{t})} |Du|^{\theta}\ dxdt
\leq 2^{n+2}\lambda^{\theta}.
\end{align*}
Recalling $N_2 = \max\left\{2^{n+2}, 1+{c_3}^{\theta}\right\}$, we obtain
$$(\tilde{y},\tilde{s}) \in \left\{ (x,t)\in K_r^\lambda(y,s):\m_{\omt}^\lambda|Du|^{\theta}\le \gh{N\lambda}^{\theta} \right\},$$
which implies \eqref{53}.

Finally, we compute from \eqref{53}, \eqref{lambda 1-1} and \eqref{co2-4} that
\begin{align*}
&\left|\left\{(x,t)\in K_r^\lambda(y,s): \m_{\omt}^\lambda|Du|^{\theta}>  \gh{N \lambda}^{\theta} \right\}\right|\\
&\qquad \leq \left|\left\{(x,t)\in K_r^\lambda(y,s):\m_{K_{4r}^\lambda(y,s)}^\lambda |Du-D\bar{v}_k|^{\theta} > \lambda^{\theta} \right\}\right|\\
&\qquad \leq  \frac{c}{\lambda^{\theta}} \int_{K_{4r}^\lambda(y,s)} |Du-D\bar{v}_k|^{\theta} \ dxdt \leq c c_4 \eta \left|Q_r^\lambda\right| < \ep \left|Q_r^\lambda\right|,
\end{align*}
by selecting $\eta$ small enough. This is a contradiction to \eqref{51}, which completes the proof.
\end{proof}

%\begin{proof}
%The proof is almost same in \cite[Lemma 5.2]{BPS21}. ????????
%\end{proof}

Taking $N=\max\mgh{N_1,N_2}$ from Lemmas \ref{DecayLem} and \ref{2_lemma}, we can apply Lemma \ref{VitCov} to discover
\begin{equation}\label{decay estimate}
\begin{aligned}
&\left|\left\{ (x,t)\in\omt: \m_{\omt}^\lambda|Du|^\theta> \gh{N\lambda}^\theta \right\}\right|\\
&\hspace{2cm} \leq \varepsilon_0 \left|\left\{(x,t)\in\omt:\m_{\omt}^\lambda|Du|^\theta > \lambda^\theta \right\}\right|\\
&\hspace{2cm}\qquad +\varepsilon_0 \left|\left\{(x,t)\in\omt:\m^\lambda_1(\mu)>\delta \lambda \right\}\right|,
\end{aligned}
\end{equation}
where $\ep_0 := \gh{\frac{80}{7}}^{n+2} \ep$.

%The intrinsic fractional maximal functions $\m_1^\la(\mu)$ are determined differently depending on the level $\la$ of the upper level set $\left\{(x,t)\in\omt:\m^\lambda_1(\mu)(x,t)>\delta \lambda \right\}$.
In the following two lemmas, we investigate the precise relation between the upper level sets of the (standard) fractional maximal function \eqref{fractional mf} and the intrinsic  one \eqref{intrinsic fractional mf} (see \cite{AH96,KS03,Min10,Phu14a} for more information about the fractional maximal function).
%This allows us to overcome the difficulty from the presence of the intrinsic fractional maximal operator.
\begin{lemma}[See {\cite[Lemmas 5.4 and 5.7]{BPS21}}]
\label{difference of ratio}
Let $\frac{2n}{n+1} < p  \le 2$ and let $\la \ge \la_0$. Then we have
\begin{equation*}\label{inclusion1}
\left\{(x,t) \in \omt : \m^\lambda_1(\mu) > \delta \lambda \right\} \subset \left\{(x,t) \in \omtb : \bgh{\m_1(\mu)}^{d} > \delta^d \lambda \right\},
\end{equation*}
where $\omtb := \OO \times (-T, T)$.
\end{lemma}

%\begin{remark}
%The difference between the standard parabolic cylinder $Q_r(x,t)$ and the intrinsic one $\iq{r}(x,t)$ causes the presence of the deficit constant $d$ in Lemma \ref{difference of ratio}. See also Remark \ref{main rmk1} \eqref{anisotropic structure}.
%\end{remark}

\begin{lemma}[See {\cite[Lemma 5.5]{BPS21}}]
\label{difference of ratio1.5}
Let $p>1$ and let $\la \ge 1$. Suppose that $\mu=\mu_0\otimes f$, where $\mu_0$ is a finite signed Radon measure in $\OO$ and $f$ is a Lebesgue function in $(-\infty,T)$. Then we have
\begin{equation*}\label{inclusion2}
\left\{(x,t) \in \omt : \m^\lambda_1(\mu) > \delta \lambda \right\}
\subset \left\{(x,t) \in \omt : \bgh{2(\m_1(\mu_0))(\m f)}^{\frac{1}{p-1}} > \delta^{\frac{1}{p-1}} \lambda \right\},
\end{equation*}
where $\m_1(\mu_0)$ is given by \eqref{fractional mf2}, and the maximal function $\m f$ is defined by
\begin{equation}\label{standard mf}
\m f (t) := \sup_{r>0} \fint_{t-r}^{t+r} |f(s)| \ ds = \sup_{r>0} \frac{1}{2r}\int_{t-r}^{t+r} |f(s)| \ ds.
\end{equation}
\end{lemma}

\begin{remark}
We note that if $\mu=\mu_0\otimes f$, we can drop the condition $\lambda_0 \ge \bgh{\frac{|\mu|(\OO_T)}{\delta T^{\frac{n+1}{2}}}}^d$ in \eqref{55}; thus we instead take
\begin{equation}\label{55-2}
\lambda_0 := \bgh{\frac{|\mu_0|(\OO)\|f\|_{L^1(0,T)}}{|\OO_T|^{\frac{n+1}{n+2}}}}^{\beta\theta} \frac{|\OO_T|}{\varepsilon \left|Q_{R/10}\right|} + 1,
\end{equation}
see \cite[Remark 5.8]{BPS21} for details.
\end{remark}

Finally, combining \eqref{decay estimate} and Lemmas \ref{difference of ratio}--\ref{difference of ratio1.5}, we directly obtain the following decay estimate:
\begin{proposition}\label{covarg}
Let $N=\max\mgh{N_1,N_2}$ from Lemmas \ref{DecayLem} and \ref{2_lemma}. Then for any $\varepsilon\in(0,1)$, there exists $\delta=\delta(n,\La_0,\La_1,p,\theta,\varepsilon) \in \left(0, \frac{1}{8}\right)$ such that if  $(\mathbf{a},\OO)$ is $(\delta,R)$-vanishing for some $R>0$, then for any renormalized solution $u$ of \eqref{pem1} and any $\lambda \ge \la_0$, we have
\begin{equation*}\label{decay1}
\begin{aligned}
\left|\left\{ (x,t)\in\omt: \m_{\omt}^\lambda|Du|^\theta > \gh{N\lambda}^\theta \right\}\right| &\leq \varepsilon_0 \left|\left\{(x,t)\in\omt:\m_{\omt}^\lambda|Du|^\theta > \lambda^\theta \right\}\right|\\
&\qquad  +\varepsilon_0 \left|\left\{(x,t)\in\omtb:\bgh{\m_1(\mu)}^{d}>\delta^d \lambda \right\}\right|.
\end{aligned}
\end{equation*}
Furthermore, if $\mu=\mu_0\otimes f$, then we have
\begin{equation*}\label{decay2}
\begin{aligned}
&\left|\left\{ (x,t)\in\omt: \m_{\omt}^\lambda|Du|^\theta > \gh{N\lambda}^\theta \right\}\right|\\
&\qquad\qquad \leq \varepsilon_0 \left|\left\{(x,t)\in\omt:\m_{\omt}^\lambda|Du|^\theta > \lambda^\theta \right\}\right|\\
&\qquad\qquad\qquad +\varepsilon_0 \left|\left\{(x,t) \in \omt : \bgh{2(\m_1(\mu_0))(\m f)}^{\frac{1}{p-1}} > \delta^{\frac{1}{p-1}} \lambda \right\}\right|.
\end{aligned}
\end{equation*}
\end{proposition}

%%%%%%%%%%%%%%%%%%%%%%%%%%%%%%%%%%%%%%%%%%%%%%%%%%%%%%%%%%%%%%%%%%%%%%%%%%%%%%%%%%%
\section{Proof of global Calder\'on-Zygmund type estimates}\label{Global gradient estimates for parabolic measure data problems}

In this section, we give the proof of Theorems \ref{main theorem} and \ref{main theorem2}. Throughout this section, we assume that $(\mathbf{a},\OO)$ is $(\delta,R)$-vanishing and that $p$ satisfies \eqref{p-range}. We fix any $\ep \in (0,1)$ and $\theta=\theta(n,p)$ satisfying \eqref{ka-range} and $\beta \theta >1$, where $\beta:=\frac{n+2}{(n+1)p-n}$. Since $\beta\theta \nearrow \frac{n+2}{n+1}$ as $\theta \nearrow p-\frac{n}{n+1}$, we can choose $\theta$ such that $\beta \theta >1$. Let $N=N(n,\La_0,\La_1,p,\theta)>1$ be a given constant in Proposition \ref{covarg} above. We assume that a renormalized solution $u$ of \eqref{pem1} is defined in $\omt : = \OO  \times (-\infty,T) $,  upon letting $u \equiv 0$ for $t < 0$ and $\mu\equiv0$ for $\mathbb{R}^{n+1}\setminus\OO_T$. If $\mu=\mu_0 \otimes f$, where $\mu_0$ is a finite signed Radon measure on $\OO$ and $f \in L^s(0,T)$ for some $s \ge 1$, then we let both $\mu_0$ and $f$ be $0$ outside $\OO$ and $(0,T)$, respectively.

We introduce the following two decay estimates of integral type:
\begin{lemma}\label{de1}
Let $\la_0\geq 1$ be a given constant in \eqref{55}.
If $u$ is a renormalized solution of \eqref{pem1}, then for any $\theta_0 \in \gh{\theta,p-\frac{n}{n+1}}$ and any $\la \ge \la_0$, there exists a constant $c=c(n,\La_0,\La_1,p,\theta_0)\ge1$ such that
\begin{equation*}\label{de1-r}
\begin{aligned}
&\int_{\left\{(x,t)\in\omt : |Du|>N\lambda\right\}} |Du|^{\theta_0} \ dxdt\\
&\qquad \leq c \varepsilon \int_{\left\{(x,t)\in\omt : |Du|>\frac{\lambda}{2}\right\}} |Du|^{\theta_0}\ dxdt\\
&\qquad\qquad+\frac{c \varepsilon}{\delta^{d \theta_0}} \int_{\left\{(x,t)\in\omtb : \bgh{\m_1(\mu)}^d>\delta^d \lambda\right\}} \bgh{\m_1(\mu)}^{d \theta_0}\ dxdt,
\end{aligned}
\end{equation*}
where $d:=\frac{2}{(n+1)p-2n}$, $\omtb := \OO\times (-T,T)$, and $\m_1(\mu)$ is given in \eqref{fractional mf}.
\end{lemma}

\begin{proof}
The proof is almost like that of \cite[Lemma 6.1]{BPS21}.
\end{proof}

\begin{lemma}\label{de2}
Let $\la_0\geq 1$ be a given constant in \eqref{55-2}.
If $u$ is a renormalized solution of \eqref{pem1} and $\mu=\mu_0\otimes f$, then for any $\theta_0 \in \gh{\theta,p-\frac{n}{n+1}}$ and any $\la \ge \la_0$, there exists a constant $c=c(n,\La_0,\La_1,p,\theta_0)\ge1$ such that
\begin{equation*}\label{de2-r}
\begin{aligned}
&\int_{\left\{(x,t)\in\omt : |Du|>N\lambda\right\}} |Du|^{\theta_0}\ dxdt\\
&\quad \leq c \varepsilon \int_{\left\{(x,t)\in\omt : |Du|>\frac{\lambda}{2}\right\}} |Du|^{\theta_0}\ dxdt \\
&\qquad +  \frac{c \varepsilon}{\delta^{\theta_0}} \int_{\left\{(x,t)\in\omt : \left[2\left(\M_1(\mu_0)\right) (\m f)\right]^\frac{1}{p-1}>\delta \lambda\right\}} \left[\left(\M_1(\mu_0) \right)(\m f)\right]^\frac{\theta_0}{p-1}\ dxdt,
\end{aligned}
\end{equation*}
where $\m_1(\mu_0)$ and $\m f$ are given in \eqref{fractional mf2} and \eqref{standard mf}, respectively.
\end{lemma}

\begin{proof}
The proof is almost like that of \cite[Lemma 6.2]{BPS21}.
\end{proof}

\subsection{Proof of Theorem \ref{main theorem}}

%We are in a position now to prove Theorems \ref{main theorem} and \ref{main theorem2}.
\begin{proof}[Proof of Theorem \ref{main theorem}]
If $0<q\leq \theta$, then we have from Lemma \ref{standard estimate} that
\begin{equation}\label{main-p1}
\integral{\OO_T}{|Du|^q}{dx dt} \le c \bgh{|\mu|(\OO_T)}^{\beta q }.
\end{equation}
Noting that
\begin{equation}\label{59}
\beta:=\frac{n+2}{(n+1)p-n} < \frac{2}{(n+1)p-2n}=:d \quad \text{for all} \  \frac{2n}{n+1}< p \le 2- \frac{1}{n+1}
\end{equation}
and $\beta_0:=\min\mgh{1,\beta q}$, we have
\begin{equation}\label{main-p1.5}
\begin{aligned}
\bgh{|\mu|(\OO_T)}^{\beta q } &= \bgh{|\mu|(\OO_T)}^{\beta_0} \bgh{|\mu|(\OO_T)}^{\beta q - \beta_0}\\
&\leq c \bgh{|\mu|(\OO_T)}^{\beta_0}\gh{\bgh{|\mu|(\OO_T)}^{d q-\beta_0}+1}\\
&= c \gh{\bgh{|\mu|(\OO_T)}^{d q }+\bgh{|\mu|(\OO_T)}^{\beta_0}}.
\end{aligned}
\end{equation}
Combining \eqref{main-p1}, \eqref{main-p1.5} and \eqref{main-p2} with $\alpha=dq$, we obtain the desired estimates \eqref{main r1-2} for the case $0<q\leq \theta$.

Now, let us assume $q >\theta $ and fix a constant $\theta_0 = \theta_0(n,p,q)$ arbitrarily so that $\theta < \theta_0 < \min\mgh{p-\frac{n}{n+1},q}$.  Recalling the truncation operator \eqref{T_k} and Lemma \ref{useful int}, we have for any $k > N\la_0$,
\begin{equation}\label{58}
\begin{aligned}
&\int_{\omt} T_k(|Du|)^{q-\theta_0}|Du|^{\theta_0}\ dxdt \\
&\qquad = (q-\theta_0) N^{q-\theta_0} \int_0^\frac{k}{N} \lambda^{q-\theta_0-1}\left[\int_{\{(x,t)\in\omt : |Du|>N\lambda\}}|Du|^{\theta_0}\ dxdt\right] d\lambda\\
&\qquad \leq c \underbrace{\int_0^{\lambda_0} \lambda^{q-\theta_0-1}\left[\int_{\{(x,t)\in\omt : |Du|>N\lambda\}}|Du|^{\theta_0}\ dxdt\right] d\lambda}_{=:P_1} \\
&\qquad  \quad\quad + c \underbrace{\int_{\lambda_0}^\frac{k}{N} \lambda^{q-\theta_0-1}\left[\int_{\{(x,t)\in\omt : |Du|>N\lambda\}}|Du|^{\theta_0}\ dxdt\right]  d\lambda}_{=:P_2}
\end{aligned}
\end{equation}
for some constant $c= c(n,\La_0,\La_1,p,q) \geq 1$, where $\lambda_0$ is given in \eqref{55}. It follows from Lemma \ref{standard estimate} that
\begin{equation*}\label{j1}
P_1 \leq \int_0^{\lambda_0} \lambda^{q-\theta_0-1}\ d\lambda \int_{\omt}|Du|^{\theta_0}\ dxdt  \leq c \lambda_0^{q-\theta_0} \bgh{|\mu|(\OO_T)}^{\be\theta_0}
\end{equation*}
for some constant $c = c(n,\La_0,p,q,\OO_T) \ge 1$. To estimate $P_2$, we apply Lemma \ref{de1} and Fubini's theorem to obtain
\begin{equation*}\label{60}
\begin{aligned}
P_2 &\leq c\varepsilon \int_{\lambda_0}^\frac{k}{N} \lambda^{q-\theta_0-1}\left[\int_{\mgh{(x,t)\in\omt : |Du|>\frac{\lambda}{2}}}|Du|^{\theta_0}\ dxdt\right] d\lambda \\
&\quad\quad + \frac{c\varepsilon}{\delta^{d \theta_0}} \int_{\lambda_0}^\frac{k}{N} \lambda^{q-\theta_0-1}\left[\int_{\mgh{(x,t)\in\omtb : \bgh{\m_1(\mu)}^d>\delta^d\lambda}}\bgh{\m_1(\mu)}^{d \theta_0}\ dxdt\right] d\lambda \\
& \leq c\ep \int_{\omt} T_k(|Du|)^{q-\theta_0}|Du|^{\theta_0}\ dxdt + \frac{c\ep}{\delta^{d q}}\int_{\omtb} \bgh{\m_1(\mu)}^{dq}\ dxdt
\end{aligned}
\end{equation*}
for some constant $c = c(n,\La_0,\La_1,p,q) \ge 1$. Inserting these inequalities into \eqref{58}, we discover
\begin{equation*}
\begin{aligned}
\int_{\omt} T_k(|Du|)^{q-\theta_0}|Du|^{\theta_0}\ dxdt &\leq c_0\varepsilon\int_{\omt} T_k(|Du|)^{q-\theta_0}|Du|^{\theta_0}\ dxdt\\
&\quad +  \frac{c\ep}{\delta^{dq}}\int_{\omtb} \bgh{\m_1(\mu)}^{dq}\ dxdt + c\lambda_0^{q-\theta_0} \bgh{|\mu|(\OO_T)}^{\be\theta_0} \\
\end{aligned}
\end{equation*}
for some $c_0 = c_0(n,\La_0,\La_1,p,q) \ge 1$. Now we select $\varepsilon>0$ with
$c_0\varepsilon<1$, and then a corresponding $\delta=\delta(n,\La_0,\La_1,p,q)>0$ is determined. Letting $k\to\infty$, we obtain
\begin{equation}\label{mr1-1}
\int_{\OO_T} |Du|^q\ dxdt \leq c \int_{\omtb} \bgh{\m_1(\mu)}^{dq}\ dxdt + c\lambda_0^{q-\theta_0} \bgh{|\mu|(\OO_T)}^{\be\theta_0}.
\end{equation}
Moreover, it follows from \eqref{55}, \eqref{ka-range} and \eqref{59} that
\begin{equation*}
\lambda_0 \le c \gh{\bgh{|\mu|(\OO_T)}^d +1}
\end{equation*}
and
\begin{equation*}
\begin{aligned}
\bgh{|\mu|(\OO_T)}^{\be\theta_0} &\le \bgh{|\mu|(\OO_T)}^{\beta_0} \gh{\bgh{|\mu|(\OO_T)}^{d} +1}^{\frac{\beta \theta_0-\beta_0}{d}}\\
&\le \bgh{|\mu|(\OO_T)}^{\beta_0} \gh{\bgh{|\mu|(\OO_T)}^{d} +1}^{\theta_0 - \frac{\beta_0}{d}},
\end{aligned}
\end{equation*}
where we used the fact that $\beta \theta_0 > \beta \theta > 1 \ge \beta_0$. Then the above two inequalities yield
\begin{equation}\label{la0}
\lambda_0^{q-\theta_0} \bgh{|\mu|(\OO_T)}^{\be\theta_0} \le c \gh{\bgh{|\mu|(\OO_T)}^{dq} + \bgh{|\mu|(\OO_T)}^{\beta_0}}
\end{equation}
for some constant $c = c(n,\La_0,\La_1,p,q,R,\OO_T) \ge 1$.
On the other hand, for each $(x,t) \in \OO_T$ we know $\m_1(\mu)(x,-t) \le \m_1(\mu)(x,t)$, which implies
\begin{equation}\label{mr1-2}
\int_{\omtb} \bgh{\m_1(\mu)}^{dq}\ dxdt \le 2 \int_{\OO_T} \bgh{\m_1(\mu)}^{dq}\ dxdt.
\end{equation}
%Combining \eqref{dc}, \eqref{mr1-1}, \eqref{la0} and \eqref{mr1-2}, we deduce for $p\ge2$ that
%\begin{equation*}\label{mr1-3}
%\int_{\OO_T} |Du|^q\ dxdt \le c \mgh{\int_{\OO_T} \bgh{\m_1(\mu)}^{q}\ dxdt + \bgh{|\mu|(\OO_T)}^{\frac{(n+2)(p-1)q}{(n+1)p-n}} +1}
%\end{equation*}
%for some $c = c(n,\La_0,\La_1,p,q,R,\OO_T) \ge 1$. Similarly, we find for $2-\frac{1}{n+1}< p \le 2$ that
%\begin{equation*}\label{mr1-4}
%\int_{\OO_T} |Du|^q\ dxdt \le c \mgh{\int_{\OO_T} \bgh{\m_1(\mu)}^{\frac{2q}{(n+1)p-2n}}\ dxdt + \bgh{|\mu|(\OO_T)}^{\frac{2q}{(n+1)p-2n}} +1}.
%\end{equation*}
%Note that $|\mu|(\OO_T) \le diam(\OO_T)^{n+1} \M_1(\mu)$, which implies
%\begin{equation*}
%\bgh{|\mu|(\OO_T)}^{\alpha q} \le c(n,\alpha,q,\OO_T) \gh{\integral{\OO_T}{\M_1(\mu)}{dxdt}}^{\alpha q}
%\end{equation*}
%for any $\alpha>0$. In light of the above estimates, we finally obtain the estimate \eqref{main r1}, which completes the proof.
Combining \eqref{main-p2} and \eqref{mr1-1}--\eqref{mr1-2}, we finally obtain the desired estimate \eqref{main r1-2} for $q>\theta$. This completes the proof.
\end{proof}

%%%%%%%%%%%%%%%%%%%%%%%
\subsection{Proof of Theorem \ref{main theorem2}}

%We are ready to prove Theorem \ref{main theorem2}.

\begin{proof}[Proof of Theorem \ref{main theorem2}]
Let $p-1<q\le \theta$. Since $\beta (p-1) <1$ and $\beta_0:=\min\mgh{1,\beta q}$, we have
\begin{equation}\label{main-p1.51}
\begin{aligned}
\bgh{|\mu|(\OO_T)}^{\beta q } &= \bgh{|\mu|(\OO_T)}^{\beta_0} \bgh{|\mu|(\OO_T)}^{\beta q - \beta_0}\\
&\leq c \bgh{|\mu|(\OO_T)}^{\beta_0}\gh{\bgh{|\mu|(\OO_T)}^{\frac{q}{p-1}-\beta_0}+1}\\
&= c \gh{\bgh{|\mu|(\OO_T)}^{\frac{q}{p-1}}+\bgh{|\mu|(\OO_T)}^{\beta_0}}.
\end{aligned}
\end{equation}
It follows from \eqref{main-p1}, \eqref{main-p1.51} and \eqref{main-p2} with $\alpha = \frac{q}{p-1}$ that
\begin{equation*}
\integral{\OO_T}{|Du|^q}{dx dt} \le c \mgh{\integral{\OO_T}{\bgh{(\m_1(\mu_0))( \m f)}^{\frac{q}{p-1}}{dxdt}}+\bgh{|\mu|(\OO_T)}^{\beta_0}}.
\end{equation*}
Here we used the fact that $\m_1 (\mu) \leq (\m_1(\mu_0))( \m f)$, where $\m_1(\mu_0)$ and $\m f$ are given in \eqref{fractional mf2} and \eqref{standard mf}, respectively.
Applying the strong $\gh{\frac{q}{p-1},\frac{q}{p-1}}$-estimate for the function $f$ (see for instance \cite[Chapter I, Theorem 1]{Ste93}), we obtain the desired estimates \eqref{main r2-2} for the case $p-1<q\leq \theta$.

Let $q >\theta $ and let $\theta_0= \theta_0(n,p,q)$ be an arbitrary constant with $\theta  < \theta_0 < \min\mgh{p-\frac{n}{n+1},q}$. Proceeding as in the proof of Theorem \ref{main theorem} and using Lemma \ref{de2}, we infer
\begin{equation}\label{mr2-1}
\begin{aligned}
\int_{\omt} T_k(|Du|)^{q-\theta_0}|Du|^{\theta_0}\ dxdt \le c \lambda_0^{q-\theta_0} \bgh{|\mu_0|(\OO)\|f\|_{L^1(0,T)}}^{\be\theta_0} + c S,
\end{aligned}
\end{equation}
where
\begin{equation*}
S := \int_{\lambda_0}^\frac{k}{N} \lambda^{q-\theta_0-1}\left[\int_{\omt\cap\mgh{ \left[2(\M_1(\mu_0))(\m f )\right]^\frac{1}{p-1} > \delta \lambda}} \left[(\M_1(\mu_0))(\m f )\right]^\frac{\theta_0}{p-1}\ dxdt\right]  d\lambda.
\end{equation*}
From Fubini's theorem and the strong $\gh{\frac{q}{p-1},\frac{q}{p-1}}$-estimate, we have
\begin{equation}\label{s3}
\begin{aligned}
S% &=  \int_{\lambda_0}^\frac{k}{N} \lambda^{q-\theta_0-1}\left[\int_{\{\left[(\M_1(\mu_0))(\m f )\right]^\frac{1}{p-1} > \delta \lambda\}} \left[(\M_1(\mu_0))(\m f )\right]^\frac{\theta_0}{p-1}\ dxdt\right] \ d\lambda \\
%& \leq \int_0^\infty \lambda^{q-\theta_0-1}\left[\int_{\{\left[(\M_1(\mu_0))(\m f )\right]^\frac{1}{p-1} > \delta \lambda\}} \left[(\M_1(\mu_0))(\m f )\right]^\frac{\theta_0}{p-1}\ dxdt\right] \ d\lambda \\
%&\leq \int_{\omt} \left[\int_0^{\frac{1}{\delta}\left[2(\M_1(\mu_0))(\m f )\right]^\frac{1}{p-1}} \lambda^{q-\theta_0-1}\ d\lambda\right] \left[(\M_1(\mu_0))(\m f )\right]^\frac{\theta_0}{p-1}\ dxdt \\
%&\leq \frac{c}{\delta^{q-\theta_0}} \int_{\omt} \left[(\M_1(\mu_0))(\m f )\right]^\frac{q}{p-1}\ dxdt \\
%&\leq \frac{c}{\delta^{q-\theta_0}} \int_{\OO} \bgh{\M_1(\mu_0)}^\frac{q}{p-1} dx \int_{-\infty}^T \left(\m f\right)^\frac{q}{p-1} dt \\
&\leq c \int_{\OO} \bgh{\M_1(\mu_0)}^\frac{q}{p-1} dx \int_0^T |f|^\frac{q}{p-1} dt,
\end{aligned}
\end{equation}
where we used the fact that $f\equiv 0$ for $t\le0$.
On the other hand, it follows from \eqref{55-2} and \eqref{ka-range} that
\begin{equation*}
\lambda_0 \le c \gh{\bgh{|\mu_0|(\OO)\|f\|_{L^1(0,T)}}^{\beta\theta} +1} \le c \gh{\bgh{|\mu_0|(\OO)\|f\|_{L^1(0,T)}}^{\frac{1}{p-1}} +1}
\end{equation*}
and
\begin{equation*}
\begin{aligned}
\bgh{|\mu|(\OO_T)}^{\be\theta_0} &\le \bgh{|\mu|(\OO_T)}^{\beta_0} \gh{\bgh{|\mu|(\OO_T)}^{\frac{1}{p-1}} +1}^{(p-1)(\beta\theta_0 - \beta_0)}\\
&\le \bgh{|\mu|(\OO_T)}^{\beta_0} \gh{\bgh{|\mu|(\OO_T)}^{\frac{1}{p-1}} +1}^{\theta_0 - \beta_0(p-1)},
\end{aligned}
\end{equation*}
where we used the fact that $\beta \theta_0 > \beta \theta >1 \ge \beta_0$ and $\beta(p-1) <1$. Since $|\mu|(\OO_T)=|\mu_0|(\OO)\|f\|_{L^1(0,T)}$, the above two inequalities imply
\begin{equation}\label{la0-2}
\begin{aligned}
&\lambda_0^{q-\theta_0} \bgh{|\mu_0|(\OO)\|f\|_{L^1(0,T)}}^{\be\theta_0}\\
&\qquad \le c \gh{\bgh{|\mu_0|(\OO)\|f\|_{L^1(0,T)}}^{\frac{q}{p-1}} + \bgh{|\mu_0|(\OO)\|f\|_{L^1(0,T)}}^{\beta_0}}.
\end{aligned}
\end{equation}
Inserting \eqref{s3} and \eqref{la0-2} into \eqref{mr2-1} and using \eqref{main-p2.6}, we discover
\begin{equation*}\label{mr2-2}
\int_{\omt} T_k(|Du|)^{q-\theta_0}|Du|^{\theta_0}\ dxdt \le c \int_{\OO_T} \left[(\M_1(\mu_0))f\right]^\frac{q}{p-1}\ dxdt + c \bgh{|\mu_0|(\OO)\|f\|_{L^1(0,T)}}^{\beta_0}.
\end{equation*}
Letting $k\to\infty$, we finally obtain the desired estimate \eqref{main r2-2} for $q>\theta$. This completes the proof.
\end{proof}

%\begin{remark}
%$$
%\theta \geq \beta
%$$
%therefore
%$$
%\int_\Omega |Du|^q\ dxdt \leq c\left(\int_{\omt} \bgh{\m_1(\mu)}^{dq}\ dxdt + |\mu|(\OO_T)^{d_0 q}+1 \right).
%$$
%and
%$$
%\int_{\omt} |Du|_k^{q-\theta_0}|Du|^{\theta_0}\ dxdt \leq c\left(\int_{\omt}\left[(\M_1(\mu_0))f\right]^\frac{q}{p-1}\ dxdt + |\mu|(\OO_T)^{q\theta }+1 \right).
%$$
%\end{remark}

%%%%%%%%%%%%%%%%%%%%%%%%%%%%%%%%%%%%%%%%%%%%%%%%%%%%%%%%%%%%%%%%%%%%%%%%%%%%%%%%%%%%%%%%%%%%

\section*{Acknowledgments}
The authors thank the anonymous referee for the valuable comments, which improved the exposition and the accuracy of the article.

%%%%%%%%%%%%%%%%%%%%%%%%%%%%%%%%%%%%%%%%%%%%%%%%%%%%%%%%%%%%%%%%%%%%%%%%%%%%%%%%%%%%%%%%%%%%

%\bibliographystyle{amsplain}

%\bibliography{JTPARK.bib}

\begin{bibdiv}
\begin{biblist}

\bib{AH96}{book}{
      author={Adams, D.~R.},
      author={Hedberg, L.~I.},
       title={Function spaces and potential theory},
      series={Grundlehren der Mathematischen Wissenschaften [Fundamental
  Principles of Mathematical Sciences]},
   publisher={Springer-Verlag, Berlin},
        date={1996},
      volume={314},
        ISBN={3-540-57060-8},
         url={http://dx.doi.org/10.1007/978-3-662-03282-4},
      review={\MR{1411441}},
}

\bib{AKP15}{article}{
      author={Avelin, B.},
      author={Kuusi, T.},
      author={Parviainen, M.},
       title={Variational parabolic capacity},
        date={2015},
        ISSN={1078-0947},
     journal={Discrete Contin. Dyn. Syst.},
      volume={35},
      number={12},
       pages={5665\ndash 5688},
         url={https://doi.org/10.3934/dcds.2015.35.5665},
      review={\MR{3393250}},
}

\bib{Bar14}{article}{
      author={Baroni, P.},
       title={Marcinkiewicz estimates for degenerate parabolic equations with
  measure data},
        date={2014},
        ISSN={0022-1236},
     journal={J. Funct. Anal.},
      volume={267},
      number={9},
       pages={3397\ndash 3426},
         url={https://doi.org/10.1016/j.jfa.2014.08.017},
      review={\MR{3261114}},
}

\bib{Bar17}{article}{
      author={Baroni, P.},
       title={Singular parabolic equations, measures satisfying density
  conditions, and gradient integrability},
        date={2017},
        ISSN={0362-546X},
     journal={Nonlinear Anal.},
      volume={153},
       pages={89\ndash 116},
         url={https://doi.org/10.1016/j.na.2016.10.019},
      review={\MR{3614663}},
}

\bib{BBGGPV95}{article}{
      author={B\'enilan, P.},
      author={Boccardo, L.},
      author={Gallou\"et, T.},
      author={Gariepy, R.},
      author={Pierre, M.},
      author={V\'azquez, J.~L.},
       title={An {$L^1$}-theory of existence and uniqueness of solutions of
  nonlinear elliptic equations},
        date={1995},
        ISSN={0391-173X},
     journal={Ann. Scuola Norm. Sup. Pisa Cl. Sci. (4)},
      volume={22},
      number={2},
       pages={241\ndash 273},
         url={http://www.numdam.org/item?id=ASNSP_1995_4_22_2_241_0},
      review={\MR{1354907}},
}

\bib{BVN15}{article}{
      author={Bidaut-V\'{e}ron, M.-F.},
      author={Nguyen, Q.-H.},
       title={Stability properties for quasilinear parabolic equations with
  measure data},
        date={2015},
        ISSN={1435-9855},
     journal={J. Eur. Math. Soc. (JEMS)},
      volume={17},
      number={9},
       pages={2103\ndash 2135},
         url={https://doi.org/10.4171/JEMS/552},
      review={\MR{3420503}},
}

\bib{BM97}{article}{
      author={Blanchard, D.},
      author={Murat, F.},
       title={Renormalised solutions of nonlinear parabolic problems with
  {$L^1$} data: existence and uniqueness},
        date={1997},
        ISSN={0308-2105},
     journal={Proc. Roy. Soc. Edinburgh Sect. A},
      volume={127},
      number={6},
       pages={1137\ndash 1152},
         url={https://doi.org/10.1017/S0308210500026986},
      review={\MR{1489429}},
}

\bib{BDGO97}{article}{
      author={Boccardo, L.},
      author={Dall'Aglio, A.},
      author={Gallou\"et, T.},
      author={Orsina, L.},
       title={Nonlinear parabolic equations with measure data},
        date={1997},
        ISSN={0022-1236},
     journal={J. Funct. Anal.},
      volume={147},
      number={1},
       pages={237\ndash 258},
         url={http://dx.doi.org/10.1006/jfan.1996.3040},
      review={\MR{1453181}},
}

\bib{BG89}{article}{
      author={Boccardo, L.},
      author={Gallou\"et, T.},
       title={Nonlinear elliptic and parabolic equations involving measure
  data},
        date={1989},
        ISSN={0022-1236},
     journal={J. Funct. Anal.},
      volume={87},
      number={1},
       pages={149\ndash 169},
         url={http://dx.doi.org/10.1016/0022-1236(89)90005-0},
      review={\MR{1025884}},
}

\bib{BG92}{article}{
      author={Boccardo, L.},
      author={Gallou\"et, T.},
       title={Nonlinear elliptic equations with right-hand side measures},
        date={1992},
        ISSN={0360-5302},
     journal={Comm. Partial Differential Equations},
      volume={17},
      number={3-4},
       pages={641\ndash 655},
         url={http://dx.doi.org/10.1080/03605309208820857},
      review={\MR{1163440}},
}

\bib{BGO96}{article}{
      author={Boccardo, L.},
      author={Gallou\"et, T.},
      author={Orsina, L.},
       title={Existence and uniqueness of entropy solutions for nonlinear
  elliptic equations with measure data},
        date={1996},
        ISSN={0294-1449},
     journal={Ann. Inst. H. Poincar\'e Anal. Non Lin\'eaire},
      volume={13},
      number={5},
       pages={539\ndash 551},
         url={http://dx.doi.org/10.1016/S0294-1449(16)30113-5},
      review={\MR{1409661}},
}

\bib{BP10}{article}{
      author={B\"ogelein, V.},
      author={Parviainen, M.},
       title={Self-improving property of nonlinear higher order parabolic
  systems near the boundary},
        date={2010},
        ISSN={1021-9722},
     journal={NoDEA Nonlinear Differential Equations Appl.},
      volume={17},
      number={1},
       pages={21\ndash 54},
         url={http://dx.doi.org/10.1007/s00030-009-0038-5},
      review={\MR{2596493}},
}

\bib{BD18}{article}{
      author={Bui, T.~A.},
      author={Duong, X.~T.},
       title={Global {M}arcinkiewicz estimates for nonlinear parabolic
  equations with nonsmooth coefficients},
        date={2018},
        ISSN={0391-173X},
     journal={Ann. Sc. Norm. Super. Pisa Cl. Sci. (5)},
      volume={18},
      number={3},
       pages={881\ndash 916},
      review={\MR{3807590}},
}

\bib{BOR13}{article}{
      author={Byun, S.-S.},
      author={Ok, J.},
      author={Ryu, S.},
       title={Global gradient estimates for general nonlinear parabolic
  equations in nonsmooth domains},
        date={2013},
        ISSN={0022-0396},
     journal={J. Differential Equations},
      volume={254},
      number={11},
       pages={4290\ndash 4326},
         url={http://dx.doi.org/10.1016/j.jde.2013.03.004},
      review={\MR{3035434}},
}

\bib{BPS20}{article}{
      author={Byun, S.-S.},
      author={Palagachev, D.~K.},
      author={Shin, P.},
       title={Optimal regularity estimates for general nonlinear parabolic
  equations},
        date={2020},
        ISSN={0025-2611},
     journal={Manuscripta Math.},
      volume={162},
      number={1-2},
       pages={67\ndash 98},
         url={https://doi.org/10.1007/s00229-019-01127-8},
      review={\MR{4079798}},
}

\bib{BPS21}{article}{
      author={Byun, S.-S.},
      author={Park, J.-T.},
      author={Shin, P.},
       title={Global regularity for degenerate/singular parabolic equations
  involving measure data},
        date={2021},
        ISSN={0944-2669},
     journal={Calc. Var. Partial Differential Equations},
      volume={60},
      number={1},
       pages={Paper No. 18, 32 pp},
         url={https://doi.org/10.1007/s00526-020-01906-2},
      review={\MR{4201641}},
}

\bib{BW04}{article}{
      author={Byun, S.-S.},
      author={Wang, L.},
       title={Elliptic equations with {BMO} coefficients in {R}eifenberg
  domains},
        date={2004},
        ISSN={0010-3640},
     journal={Comm. Pure Appl. Math.},
      volume={57},
      number={10},
       pages={1283\ndash 1310},
         url={http://dx.doi.org/10.1002/cpa.20037},
      review={\MR{2069724}},
}

\bib{CP98}{article}{
      author={Caffarelli, L.~A.},
      author={Peral, I.},
       title={On {$W^{1,p}$} estimates for elliptic equations in divergence
  form},
        date={1998},
        ISSN={0010-3640},
     journal={Comm. Pure Appl. Math.},
      volume={51},
      number={1},
       pages={1\ndash 21},
  url={http://dx.doi.org/10.1002/(SICI)1097-0312(199801)51:1<1::AID-CPA1>3.3.CO;2-N},
      review={\MR{1486629}},
}

\bib{Cas86}{article}{
      author={Casas, E.},
       title={Control of an elliptic problem with pointwise state constraints},
        date={1986},
        ISSN={0363-0129},
     journal={SIAM J. Control Optim.},
      volume={24},
      number={6},
       pages={1309\ndash 1318},
         url={https://doi.org/10.1137/0324078},
      review={\MR{861100}},
}

\bib{Cas93}{article}{
      author={Casas, E.},
       title={Boundary control of semilinear elliptic equations with pointwise
  state constraints},
        date={1993},
        ISSN={0363-0129},
     journal={SIAM J. Control Optim.},
      volume={31},
      number={4},
       pages={993\ndash 1006},
         url={http://dx.doi.org/10.1137/0331044},
      review={\MR{1227543}},
}

\bib{CdT08}{article}{
      author={Casas, E.},
      author={de~los Reyes, J.~C.},
      author={Tr\"oltzsch, F.},
       title={Sufficient second-order optimality conditions for semilinear
  control problems with pointwise state constraints},
        date={2008},
        ISSN={1052-6234},
     journal={SIAM J. Optim.},
      volume={19},
      number={2},
       pages={616\ndash 643},
         url={http://dx.doi.org/10.1137/07068240X},
      review={\MR{2425032}},
}

\bib{DMOP99}{article}{
      author={Dal~Maso, G.},
      author={Murat, F.},
      author={Orsina, L.},
      author={Prignet, A.},
       title={Renormalized solutions of elliptic equations with general measure
  data},
        date={1999},
        ISSN={0391-173X},
     journal={Ann. Scuola Norm. Sup. Pisa Cl. Sci. (4)},
      volume={28},
      number={4},
       pages={741\ndash 808},
         url={http://www.numdam.org/item?id=ASNSP_1999_4_28_4_741_0},
      review={\MR{1760541}},
}

\bib{Dal96}{article}{
      author={Dall'Aglio, A.},
       title={Approximated solutions of equations with {$L^1$} data.
  {A}pplication to the {$H$}-convergence of quasi-linear parabolic equations},
        date={1996},
        ISSN={0003-4622},
     journal={Ann. Mat. Pura Appl. (4)},
      volume={170},
       pages={207\ndash 240},
         url={http://dx.doi.org/10.1007/BF01758989},
      review={\MR{1441620}},
}

\bib{DiB93}{book}{
      author={DiBenedetto, E.},
       title={Degenerate parabolic equations},
      series={Universitext},
   publisher={Springer-Verlag, New York},
        date={1993},
        ISBN={0-387-94020-0},
         url={https://doi.org/10.1007/978-1-4612-0895-2},
      review={\MR{1230384}},
}

\bib{DF85}{article}{
      author={DiBenedetto, E.},
      author={Friedman, A.},
       title={H\"{o}lder estimates for nonlinear degenerate parabolic systems},
        date={1985},
        ISSN={0075-4102},
     journal={J. Reine Angew. Math.},
      volume={357},
       pages={1\ndash 22},
         url={https://doi.org/10.1515/crll.1985.357.1},
      review={\MR{783531}},
}

\bib{DF85b}{article}{
      author={DiBenedetto, E.},
      author={Friedman, A.},
       title={Addendum to: ``{H}\"{o}lder estimates for nonlinear degenerate
  parabolic systems''},
        date={1985},
        ISSN={0075-4102},
     journal={J. Reine Angew. Math.},
      volume={363},
       pages={217\ndash 220},
         url={https://doi.org/10.1515/crll.1985.363.217},
      review={\MR{814022}},
}

\bib{DL89b}{article}{
      author={DiPerna, R.~J.},
      author={Lions, P.-L.},
       title={On the {C}auchy problem for {B}oltzmann equations: global
  existence and weak stability},
        date={1989},
        ISSN={0003-486X},
     journal={Ann. of Math. (2)},
      volume={130},
      number={2},
       pages={321\ndash 366},
         url={https://doi.org/10.2307/1971423},
      review={\MR{1014927}},
}

\bib{DL89a}{article}{
      author={DiPerna, R.~J.},
      author={Lions, P.-L.},
       title={Ordinary differential equations, transport theory and {S}obolev
  spaces},
        date={1989},
        ISSN={0020-9910},
     journal={Invent. Math.},
      volume={98},
      number={3},
       pages={511\ndash 547},
         url={https://doi.org/10.1007/BF01393835},
      review={\MR{1022305}},
}

\bib{DZ21b}{misc}{
      author={Dong, H.},
      author={Zhu, H.},
       title={Gradient estimates for singular parabolic $p$-laplace type
  equations with measure data},
        date={2021},
        note={\href{https://arxiv.org/abs/2111.03050}{arXiv:2111.03050}},
}

\bib{DPP03}{article}{
      author={Droniou, J.},
      author={Porretta, A.},
      author={Prignet, A.},
       title={Parabolic capacity and soft measures for nonlinear equations},
        date={2003},
        ISSN={0926-2601},
     journal={Potential Anal.},
      volume={19},
      number={2},
       pages={99\ndash 161},
         url={https://doi.org/10.1023/A:1023248531928},
      review={\MR{1976292}},
}

\bib{DM10}{article}{
      author={Duzaar, F.},
      author={Mingione, G.},
       title={Gradient estimates via linear and nonlinear potentials},
        date={2010},
        ISSN={0022-1236},
     journal={J. Funct. Anal.},
      volume={259},
      number={11},
       pages={2961\ndash 2998},
         url={http://dx.doi.org/10.1016/j.jfa.2010.08.006},
      review={\MR{2719282}},
}

\bib{DM11}{article}{
      author={Duzaar, F.},
      author={Mingione, G.},
       title={Gradient estimates via non-linear potentials},
        date={2011},
        ISSN={0002-9327},
     journal={Amer. J. Math.},
      volume={133},
      number={4},
       pages={1093\ndash 1149},
         url={http://dx.doi.org/10.1353/ajm.2011.0023},
      review={\MR{2823872}},
}

\bib{FST91}{article}{
      author={Fukushima, M.},
      author={Sato, K.-i.},
      author={Taniguchi, S.},
       title={On the closable parts of pre-{D}irichlet forms and the fine
  supports of underlying measures},
        date={1991},
        ISSN={0030-6126},
     journal={Osaka J. Math.},
      volume={28},
      number={3},
       pages={517\ndash 535},
         url={http://projecteuclid.org/euclid.ojm/1200783223},
      review={\MR{1144471}},
}

\bib{Giu03}{book}{
      author={Giusti, E.},
       title={Direct methods in the calculus of variations},
   publisher={World Scientific Publishing Co., Inc., River Edge, NJ},
        date={2003},
        ISBN={981-238-043-4},
         url={http://dx.doi.org/10.1142/9789812795557},
      review={\MR{1962933}},
}

\bib{Ham92}{article}{
      author={Hamburger, C.},
       title={Regularity of differential forms minimizing degenerate elliptic
  functionals},
        date={1992},
        ISSN={0075-4102},
     journal={J. Reine Angew. Math.},
      volume={431},
       pages={7\ndash 64},
         url={https://doi.org/10.1515/crll.1992.431.7},
      review={\MR{1179331}},
}

\bib{KKKP13}{article}{
      author={Kinnunen, J.},
      author={Korte, R.},
      author={Kuusi, T.},
      author={Parviainen, M.},
       title={Nonlinear parabolic capacity and polar sets of superparabolic
  functions},
        date={2013},
        ISSN={0025-5831},
     journal={Math. Ann.},
      volume={355},
      number={4},
       pages={1349\ndash 1381},
         url={https://doi.org/10.1007/s00208-012-0825-x},
      review={\MR{3037018}},
}

\bib{KL00}{article}{
      author={Kinnunen, J.},
      author={Lewis, J.~L.},
       title={Higher integrability for parabolic systems of {$p$}-{L}aplacian
  type},
        date={2000},
        ISSN={0012-7094},
     journal={Duke Math. J.},
      volume={102},
      number={2},
       pages={253\ndash 271},
         url={http://dx.doi.org/10.1215/S0012-7094-00-10223-2},
      review={\MR{1749438}},
}

\bib{KLP10}{article}{
      author={Kinnunen, J.},
      author={Lukkari, T.},
      author={Parviainen, M.},
       title={An existence result for superparabolic functions},
        date={2010},
        ISSN={0022-1236},
     journal={J. Funct. Anal.},
      volume={258},
      number={3},
       pages={713\ndash 728},
         url={http://dx.doi.org/10.1016/j.jfa.2009.08.009},
      review={\MR{2558174}},
}

\bib{KLP13}{article}{
      author={Kinnunen, J.},
      author={Lukkari, T.},
      author={Parviainen, M.},
       title={Local approximation of superharmonic and superparabolic functions
  in nonlinear potential theory},
        date={2013},
        ISSN={1661-7738},
     journal={J. Fixed Point Theory Appl.},
      volume={13},
      number={1},
       pages={291\ndash 307},
         url={https://doi.org/10.1007/s11784-013-0108-5},
      review={\MR{3071955}},
}

\bib{KS03}{article}{
      author={Kinnunen, J.},
      author={Saksman, E.},
       title={Regularity of the fractional maximal function},
        date={2003},
        ISSN={0024-6093},
     journal={Bull. London Math. Soc.},
      volume={35},
      number={4},
       pages={529\ndash 535},
         url={https://doi.org/10.1112/S0024609303002017},
      review={\MR{1979008}},
}

\bib{KR19}{article}{
      author={Klimsiak, T.},
      author={Rozkosz, A.},
       title={On the structure of diffuse measures for parabolic capacities},
        date={2019},
        ISSN={1631-073X},
     journal={C. R. Math. Acad. Sci. Paris},
      volume={357},
      number={5},
       pages={443\ndash 449},
         url={https://doi.org/10.1016/j.crma.2019.04.012},
      review={\MR{3959840}},
}

\bib{KM13b}{article}{
      author={Kuusi, T.},
      author={Mingione, G.},
       title={Gradient regularity for nonlinear parabolic equations},
        date={2013},
        ISSN={0391-173X},
     journal={Ann. Sc. Norm. Super. Pisa Cl. Sci. (5)},
      volume={12},
      number={4},
       pages={755\ndash 822},
      review={\MR{3184569}},
}

\bib{KM14c}{article}{
      author={Kuusi, T.},
      author={Mingione, G.},
       title={Riesz potentials and nonlinear parabolic equations},
        date={2014},
        ISSN={0003-9527},
     journal={Arch. Ration. Mech. Anal.},
      volume={212},
      number={3},
       pages={727\ndash 780},
         url={https://doi.org/10.1007/s00205-013-0695-8},
      review={\MR{3187676}},
}

\bib{KM14b}{article}{
      author={Kuusi, T.},
      author={Mingione, G.},
       title={The {W}olff gradient bound for degenerate parabolic equations},
        date={2014},
        ISSN={1435-9855},
     journal={J. Eur. Math. Soc. (JEMS)},
      volume={16},
      number={4},
       pages={835\ndash 892},
         url={http://dx.doi.org/10.4171/JEMS/449},
      review={\MR{3191979}},
}

\bib{KM18}{article}{
      author={Kuusi, T.},
      author={Mingione, G.},
       title={Vectorial nonlinear potential theory},
        date={2018},
        ISSN={1435-9855},
     journal={J. Eur. Math. Soc. (JEMS)},
      volume={20},
      number={4},
       pages={929\ndash 1004},
         url={https://doi.org/10.4171/JEMS/780},
      review={\MR{3779689}},
}

\bib{LMS14}{article}{
      author={Lemenant, A.},
      author={Milakis, E.},
      author={Spinolo, L.~V.},
       title={On the extension property of {R}eifenberg-flat domains},
        date={2014},
        ISSN={1239-629X},
     journal={Ann. Acad. Sci. Fenn. Math.},
      volume={39},
      number={1},
       pages={51\ndash 71},
         url={http://dx.doi.org/10.5186/aasfm.2014.3907},
      review={\MR{3186805}},
}

\bib{LL94}{article}{
      author={LeVeque, R.~J.},
      author={Li, Z.},
       title={The immersed interface method for elliptic equations with
  discontinuous coefficients and singular sources},
        date={1994},
        ISSN={0036-1429},
     journal={SIAM J. Numer. Anal.},
      volume={31},
      number={4},
       pages={1019\ndash 1044},
         url={https://doi.org/10.1137/0731054},
      review={\MR{1286215}},
}

\bib{Lie93}{article}{
      author={Lieberman, G.~M.},
       title={Boundary and initial regularity for solutions of degenerate
  parabolic equations},
        date={1993},
        ISSN={0362-546X},
     journal={Nonlinear Anal.},
      volume={20},
      number={5},
       pages={551\ndash 569},
         url={https://doi.org/10.1016/0362-546X(93)90038-T},
      review={\MR{1207530}},
}

\bib{MPS11}{article}{
      author={Meyer, C.},
      author={Panizzi, L.},
      author={Schiela, A.},
       title={Uniqueness criteria for the adjoint equation in state-constrained
  elliptic optimal control},
        date={2011},
        ISSN={0163-0563},
     journal={Numer. Funct. Anal. Optim.},
      volume={32},
      number={9},
       pages={983\ndash 1007},
         url={http://dx.doi.org/10.1080/01630563.2011.587074},
      review={\MR{2823475}},
}

\bib{Min07}{article}{
      author={Mingione, G.},
       title={The {C}alder\'on-{Z}ygmund theory for elliptic problems with
  measure data},
        date={2007},
        ISSN={0391-173X},
     journal={Ann. Sc. Norm. Super. Pisa Cl. Sci. (5)},
      volume={6},
      number={2},
       pages={195\ndash 261},
      review={\MR{2352517}},
}

\bib{Min10}{article}{
      author={Mingione, G.},
       title={Gradient estimates below the duality exponent},
        date={2010},
        ISSN={0025-5831},
     journal={Math. Ann.},
      volume={346},
      number={3},
       pages={571\ndash 627},
         url={http://dx.doi.org/10.1007/s00208-009-0411-z},
      review={\MR{2578563}},
}

\bib{Ngu_pre}{article}{
      author={Nguyen, Q.-H.},
       title={Potential estimates and quasilinear parabolic equations with
  measure data},
     journal={Mem. Amer. Math. Soc.},
       pages={to appear},
        note={\href{https://arxiv.org/abs/1405.2587}{arXiv:1405.2587}},
}

\bib{NP19}{article}{
      author={Nguyen, Q.-H.},
      author={Phuc, N.~C.},
       title={Good-{$\lambda$} and {M}uckenhoupt-{W}heeden type bounds in
  quasilinear measure datum problems, with applications},
        date={2019},
        ISSN={0025-5831},
     journal={Math. Ann.},
      volume={374},
      number={1-2},
       pages={67\ndash 98},
         url={https://doi.org/10.1007/s00208-018-1744-2},
      review={\MR{3961305}},
}

\bib{NP20}{article}{
      author={Nguyen, Q.-H.},
      author={Phuc, N.~C.},
       title={Pointwise gradient estimates for a class of singular quasilinear
  equations with measure data},
        date={2020},
        ISSN={0022-1236},
     journal={J. Funct. Anal.},
      volume={278},
      number={5},
       pages={108391, 35 pp},
         url={https://doi.org/10.1016/j.jfa.2019.108391},
      review={\MR{4046205}},
}

\bib{NP20b}{article}{
      author={Nguyen, Q.-H.},
      author={Phuc, N.~C.},
       title={Existence and regularity estimates for quasilinear equations with
  measure data: the case $1<p\leq \frac{3n-2}{2n-1}$},
    journal={Anal. PDE},
      pages={to appear},
        note={\href{https://arxiv.org/abs/2003.03725}{arXiv:2003.03725}},
}

\bib{OF03}{book}{
      author={Osher, S.},
      author={Fedkiw, R.},
       title={Level set methods and dynamic implicit surfaces},
      series={Applied Mathematical Sciences},
   publisher={Springer-Verlag, New York},
        date={2003},
      volume={153},
        ISBN={0-387-95482-1},
         url={https://doi.org/10.1007/b98879},
      review={\MR{1939127}},
}

\bib{Pes77}{article}{
      author={Peskin, C.~S.},
       title={Numerical analysis of blood flow in the heart},
        date={1977},
        ISSN={0021-9991},
     journal={J. Computational Phys.},
      volume={25},
      number={3},
       pages={220\ndash 252},
      review={\MR{0490027}},
}

\bib{PM89}{article}{
      author={Peskin, C.~S.},
      author={McQueen, D.~M.},
       title={A three-dimensional computational method for blood flow in the
  heart. {I}. {I}mmersed elastic fibers in a viscous incompressible fluid},
        date={1989},
        ISSN={0021-9991},
     journal={J. Comput. Phys.},
      volume={81},
      number={2},
       pages={372\ndash 405},
         url={https://doi.org/10.1016/0021-9991(89)90213-1},
      review={\MR{994353}},
}

\bib{Pet08}{article}{
      author={Petitta, F.},
       title={Renormalized solutions of nonlinear parabolic equations with
  general measure data},
        date={2008},
        ISSN={0373-3114},
     journal={Ann. Mat. Pura Appl. (4)},
      volume={187},
      number={4},
       pages={563\ndash 604},
         url={http://dx.doi.org/10.1007/s10231-007-0057-y},
      review={\MR{2413369}},
}

\bib{PPP11}{article}{
      author={Petitta, F.},
      author={Ponce, A.~C.},
      author={Porretta, A.},
       title={Diffuse measures and nonlinear parabolic equations},
        date={2011},
        ISSN={1424-3199},
     journal={J. Evol. Equ.},
      volume={11},
      number={4},
       pages={861\ndash 905},
         url={https://doi.org/10.1007/s00028-011-0115-1},
      review={\MR{2861310}},
}

\bib{PP15}{article}{
      author={Petitta, F.},
      author={Porretta, A.},
       title={On the notion of renormalized solution to nonlinear parabolic
  equations with general measure data},
        date={2015},
        ISSN={2296-9020},
     journal={J. Elliptic Parabol. Equ.},
      volume={1},
       pages={201\ndash 214},
         url={https://doi.org/10.1007/BF03377376},
      review={\MR{3403419}},
}

\bib{Phu14c}{article}{
      author={Phuc, N.~C.},
       title={Global integral gradient bounds for quasilinear equations below
  or near the natural exponent},
        date={2014},
        ISSN={0004-2080},
     journal={Ark. Mat.},
      volume={52},
      number={2},
       pages={329\ndash 354},
         url={http://dx.doi.org/10.1007/s11512-012-0177-5},
      review={\MR{3255143}},
}

\bib{Phu14b}{article}{
      author={Phuc, N.~C.},
       title={Morrey global bounds and quasilinear {R}iccati type equations
  below the natural exponent},
        date={2014},
        ISSN={0021-7824},
     journal={J. Math. Pures Appl. (9)},
      volume={102},
      number={1},
       pages={99\ndash 123},
         url={http://dx.doi.org/10.1016/j.matpur.2013.11.003},
      review={\MR{3212250}},
}

\bib{Phu14a}{article}{
      author={Phuc, N.~C.},
       title={Nonlinear {M}uckenhoupt-{W}heeden type bounds on {R}eifenberg
  flat domains, with applications to quasilinear {R}iccati type equations},
        date={2014},
        ISSN={0001-8708},
     journal={Adv. Math.},
      volume={250},
       pages={387\ndash 419},
         url={http://dx.doi.org/10.1016/j.aim.2013.09.022},
      review={\MR{3122172}},
}

\bib{Pie83}{article}{
      author={Pierre, M.},
       title={Parabolic capacity and {S}obolev spaces},
        date={1983},
        ISSN={0036-1410},
     journal={SIAM J. Math. Anal.},
      volume={14},
      number={3},
       pages={522\ndash 533},
         url={https://doi.org/10.1137/0514044},
      review={\MR{697527}},
}

\bib{Pri97}{article}{
      author={Prignet, A.},
       title={Existence and uniqueness of ``entropy'' solutions of parabolic
  problems with {$L^1$} data},
        date={1997},
        ISSN={0362-546X},
     journal={Nonlinear Anal.},
      volume={28},
      number={12},
       pages={1943\ndash 1954},
         url={https://doi.org/10.1016/S0362-546X(96)00030-2},
      review={\MR{1436364}},
}

\bib{Sch12b}{article}{
      author={Scheven, C.},
       title={Elliptic obstacle problems with measure data: potentials and low
  order regularity},
        date={2012},
        ISSN={0214-1493},
     journal={Publ. Mat.},
      volume={56},
      number={2},
       pages={327\ndash 374},
         url={https://doi.org/10.5565/PUBLMAT_56212_04},
      review={\MR{2978327}},
}

\bib{Sch12a}{article}{
      author={Scheven, C.},
       title={Gradient potential estimates in non-linear elliptic obstacle
  problems with measure data},
        date={2012},
        ISSN={0022-1236},
     journal={J. Funct. Anal.},
      volume={262},
      number={6},
       pages={2777\ndash 2832},
         url={https://doi.org/10.1016/j.jfa.2012.01.003},
      review={\MR{2885965}},
}

\bib{Ste93}{book}{
      author={Stein, E.~M.},
       title={Harmonic analysis: real-variable methods, orthogonality, and
  oscillatory integrals},
      series={Princeton Mathematical Series},
   publisher={Princeton University Press, Princeton, NJ},
        date={1993},
      volume={43},
        ISBN={0-691-03216-5},
        note={With the assistance of Timothy S. Murphy, Monographs in Harmonic
  Analysis, III},
      review={\MR{1232192}},
}

\bib{SSO94}{article}{
      author={Sussman, M.},
      author={Smereka, P.},
      author={Osher, S.},
       title={A level set approach for computing solutions to incompressible
  two-phase flow},
        date={1994},
        ISSN={0021-9991},
     journal={J. Comput. Phys.},
      volume={114},
      number={1},
       pages={146\ndash 159},
         url={http://dx.doi.org/10.1006/jcph.1994.1155},
}

\bib{Tor97}{article}{
      author={Toro, T.},
       title={Doubling and flatness: geometry of measures},
        date={1997},
        ISSN={0002-9920},
     journal={Notices Amer. Math. Soc.},
      volume={44},
      number={9},
       pages={1087\ndash 1094},
      review={\MR{1470167}},
}

\bib{Urb08}{book}{
      author={Urbano, J.~M.},
       title={The method of intrinsic scaling},
      series={Lecture Notes in Mathematics},
   publisher={Springer-Verlag, Berlin},
        date={2008},
      volume={1930},
        ISBN={978-3-540-75931-7},
         url={https://doi.org/10.1007/978-3-540-75932-4},
        note={A systematic approach to regularity for degenerate and singular
  PDEs},
      review={\MR{2451216}},
}

\bib{Wan03}{article}{
      author={Wang, L.},
       title={A geometric approach to the {C}alder\'on-{Z}ygmund estimates},
        date={2003},
        ISSN={1439-8516},
     journal={Acta Math. Sin. (Engl. Ser.)},
      volume={19},
      number={2},
       pages={381\ndash 396},
         url={http://dx.doi.org/10.1007/s10114-003-0264-4},
      review={\MR{1987802}},
}

\end{biblist}
\end{bibdiv}

% \bib, bibdiv, biblist are defined by the amsrefs package.

\end{document}